\documentclass[10pt]{amsart}
\usepackage{amsfonts,amscd,amsthm,amsgen,amsmath,amssymb}
\usepackage{epstopdf}
\usepackage{epsfig,color}
\usepackage[all]{xy}
\usepackage[vcentermath]{youngtab}
\usepackage{tikz}
\usetikzlibrary{decorations.markings}
\usepackage[letterpaper, hmargin=1.3in]{geometry}
\usepackage{chngpage}
\usepackage{booktabs}
\newtheorem{theorem}{Theorem}[section]
\newtheorem{lemma}[theorem]{Lemma}

\newtheorem{proposition}[theorem]{Proposition}

\theoremstyle{definition}
\newtheorem{definition}[theorem]{Definition}

\theoremstyle{remark}
\newtheorem{remark}[theorem]{Remark}

\newcommand{\Aff}{A\hspace{-1pt}f\hspace{-2pt}f}

\numberwithin{equation}{section}

\begin{document}
\tikzset{->-/.style={decoration={
  markings,
  mark=at position #1 with {\arrow{>}}},postaction={decorate}}}

\tikzset{-<-/.style={decoration={
  markings,
  mark=at position #1 with {\arrow{<}}},postaction={decorate}}}

\title[Arithmetic of singular character varieties and their $E$-polynomials]{Arithmetic of singular character varieties and their $E$-polynomials}

\author{David Baraglia}

\address{School of Mathematical Sciences, The University of Adelaide, Adelaide SA 5005, Australia}
\email{david.baraglia@adelaide.edu.au}

\author{Pedram Hekmati}

\address{Instituto Nacional de Matem\'atica Pura e Aplicada, Estrada Dona Castorina 110, Rio de Janeiro 22460-320, Brazil}
\email{pedram.hekmati@adelaide.edu.au}

\begin{abstract}
We calculate the $E$-polynomials of the $SL_3(\mathbb{C})$ and $GL_3(\mathbb{C})$-character varieties of compact oriented surfaces of any genus and the $E$-polynomials of the $SL_2(\mathbb{C})$ and $GL_2(\mathbb{C})$-character varieties of compact non-orientable surfaces of any Euler characteristic. Our methods also give a new and significantly simpler computation of the $E$-polynomials of the $SL_2(\mathbb{C})$-character varieties of compact orientable surfaces, which were computed by Logares, Mu\~noz and Newstead for genus $g=1,2$ and by Martinez and Mu\~noz for $g \ge 3$. Our technique is based on the arithmetic of character varieties over finite fields. More specifically, we show how to extend the approach of Hausel and Rodriguez-Villegas used for non-singular (twisted) character varieties to the singular (untwisted) case.
\end{abstract}



\date{\today}


\maketitle


\section{Introduction}

Let $G$ be a complex reductive group and $\Gamma$ a finitely generated group. The character variety $Rep(\Gamma , G)$ is the moduli space of reductive representations of $\Gamma$ into $G$. In this paper we introduce techniques for computing the $E$-polynomials of a number of such character varieties. Recall that for any complex algebraic variety, the $E$-polynomial encodes its virtual Hodge numbers defined with respect to its canonical mixed Hodge structure on compactly supported cohomology.\\

The pioneering work in this subject is the paper of Hausel and Rodriguez-Villegas \cite{harv}, which uses arithmetic techniques to compute the $E$-polynomials of certain non-singular (twisted) character varieties. In this paper, we will show how to extend the approach of Hausel and Rodriguez-Villegas to the case of untwisted character varieties, which are almost always singular, due to the presence of reducible representations. We then proceed to apply our method to compute the $E$-polynomials of a number of such character varieties, as detailed in Section \ref{seccompu}. Amongst other results, we compute the $E$-polynomials of the $SL_3(\mathbb{C})$ and $GL_3(\mathbb{C})$-character varieties for $\Gamma$ the fundamental group of a compact oriented surface (Theorem \ref{thmcos3}) and the $E$-polynomials of the $SL_2(\mathbb{C})$ and $GL_2(\mathbb{C})$-character varieties for $\Gamma$ the fundamental group of a compact non-orientable surface (Theorem \ref{thmcnos2}). We also recover the formula in \cite{mamu2} for $E$-polynomial of the $SL_2(\mathbb{C})$-character variety for the fundamental group of a compact oriented surface, through a much simpler calculation (Theorem \ref{thmcos2}).\\

In this paper $G$ will always be either a general linear group $GL_n(\mathbb{C})$ or a special linear group $SL_n(\mathbb{C})$, though our methods could certainly be applied to other reductive groups. Our first main result is Theorem \ref{thmcounting}, which reduces the computation of the $E$-polynomials of the character varieties $Rep(\Gamma , GL_n(\mathbb{C}))$ and $Rep(\Gamma , SL_n(\mathbb{C}))$ to a problem in arithmetic:

\begin{theorem}\label{thmcounting}
Suppose there is a polynomial $A(t) \in \mathbb{\mathbb{C}}[t]$ and a positive integer $N$ such that for every finite field $\mathbb{F}_q$ of order $q$ with $q = 1 \; ({\rm mod} \; N)$, the number of isomorphism classes of $n$-dimensional reductive representations over $\mathbb{F}_q$ (resp. $n$-dimensional reductive representations over $\mathbb{F}_q$ with trivial determinant) equals $A(q)$. Then $A(q)$ is the $E$-polynomial of the complex character variety $Rep( \Gamma , GL_n(\mathbb{C}))$ (resp. $Rep(\Gamma , SL_n(\mathbb{C}))$). 
\end{theorem}

Thus to calculate the $E$-polynomials of the $GL_n(\mathbb{C})$ and $SL_n(\mathbb{C})$-character varieties of $\Gamma$, we need to count $n$-dimensional reductive representations of $\Gamma$ over finite fields, or $n$-dimensional representations with fixed determinant. Our strategy for doing this is as follows. Let $Hom(\Gamma , GL_n(\mathbb{F}_q))$ (resp. $Hom(\Gamma , SL_n(\mathbb{F}_q))$ ) be the set of all homomorphisms from $\Gamma$ into $GL_n(\mathbb{F}_q)$ (resp. $SL_n(\mathbb{F}_q)$). Following \cite{harv}, \cite{mer} we will compute the number of elements of $Hom(\Gamma , GL_n(\mathbb{F}_q))$ and $Hom(\Gamma , SL_n(\mathbb{F}_q))$ using the character theory of the general linear and special linear groups over finite fields. We then count the number of {\em reducible} homomorphisms of $\Gamma$ into $GL_n(\mathbb{F}_q)$ and $SL_n(\mathbb{F}_q)$. Taking all representations and subtracting off the reducible representations gives the number of irreducible representations, from which one can readily obtain the number of reductive representations. In principle this gives a formula, recursive in $n$, for the number of $n$-dimensional representations, at least for certain groups $\Gamma$. However, the computations rapidly increase in complexity with $n$, so in this paper we will only carry out the explicit computations in the cases $n=2,3$.

\subsection{Computations}\label{seccompu}
We state the results of our $E$-polynomial computations. For a given group $\Gamma$ we denote by $A_{GL_n}(q)$ the number of $n$-dimensional reductive representations of $\Gamma$ over finite fields $\mathbb{F}_q$ of order $q$ satisfying $q = 1 \; ({\rm mod} \; N)$, where $N$ is as in Theorem \ref{thmcounting}. Similarly we denote by $A_{SL_n}(q)$ the number of $n$-dimensional reductive representations of $\Gamma$ with trivial determinant. By Theorem \ref{thmcounting}, these will also give the $E$-polynomials of the corresponding complex character varieties.

\begin{theorem}[Free groups]
Let $\Gamma = F_r$ be the free group on $r$ generators. Then:

\begin{equation*}
\begin{aligned}
A_{GL_2}(q) &= (q-1)^r \left(  (q^3-q)^{r-1} -(q^2-q)^{r-1} + q\left( \frac{1}{2}(q+1)^{r-1} + \frac{1}{2}(q-1)^{r-1}\right) \right). \\
A_{SL_2}(q) &= (q^3-q)^{r-1} - (q^2-q)^{r-1} + q\left( \frac{1}{2}(q+1)^{r-1} + \frac{1}{2}(q-1)^{r-1} \right).
\end{aligned}
\end{equation*}

\end{theorem}
The $E$-polynomial for the $SL_2(\mathbb{C})$ case was computed in \cite{cala} and \cite{more}. The $GL_2(\mathbb{C})$ case was also shown in \cite{more}.

\begin{theorem}[Compact oriented surfaces, $n=2$]\label{thmcos2}
Let $\Gamma = \pi_1(\Sigma_g)$ be the fundamental group of a compact oriented surface $\Sigma_g$ of genus $g$. Then:

\begin{equation*}
\begin{aligned}
A_{GL_2}(q) &= (q-1)^{2g} \left( (q^3-q)^{2g-2} + (q^2-1)^{2g-2} + q\left( \frac{1}{2}(q^2+q)^{2g-2} + \frac{1}{2}(q^2-q)^{2g-2} \right) \right. \\
& \; \; \; \; \; \left. -q(q^2-q)^{2g-2} -q^{2g-2} + q\left( \frac{1}{2}(q+1)^{2g-1} + \frac{1}{2}(q-1)^{2g-1} \right) \right).
\end{aligned}
\end{equation*}
\begin{equation*}
\begin{aligned}
A_{SL_2}(q) &= (q^3-q)^{2g-2} + (q^2-1)^{2g-2} - q(q^2-q)^{2g-2} - 2^{2g}q^{2g-2} \\
& \; \; \; \; \; + \frac{(q-1)}{2}\left(  (q^2+q)^{2g-2} + (q^2-q)^{2g-2} \right) + 2^{2g-1}\left(  (q^2+q)^{2g-2} + (q^2-q)^{2g-2} \right) \\
& \; \; \; \; \; + q\left( \frac{1}{2}(q+1)^{2g-1} + \frac{1}{2}(q-1)^{2g-1} \right).
\end{aligned}
\end{equation*}

\end{theorem}
The $E$-polynomial for $SL_2(\mathbb{C})$ was first computed for the cases $g=1,2$ in \cite{lomune}, $g=3$ in \cite{mamu1} and $g > 3$ in \cite{mamu2}. We note that the computation in \cite{mamu2} for the $g>3$ case uses the $g=1,2,3$ cases as inputs, so the complete result depends on all three of these papers. Our computation, based on entirely different methods turns out to be substantially simpler. To the best of our knowledge the $E$-polynomial for the corresponding $GL_2(\mathbb{C})$-character varieties have not previously been computed.

\begin{theorem}[Compact non-orientable surfaces]\label{thmcnos2}
Let $\Gamma = \pi_1(\Sigma_k)$ be the fundamental group of a compact non-orientable surface $\Sigma_k$ of Euler characteristic $2-k$. Then:

\begin{equation*}
\begin{aligned}
\frac{A_{GL_2}(q)}{(q-1)^{k-1}} &= \frac{(q-1)}{2}(q^2-q)^{k-2} +\frac{(q-1)}{2}(q^2+q)^{k-2} + 2(q^3-q)^{k-2} + 2(q^2-1)^{k-2} \\
& \; \; \; \; \; q\left( (q+1)^{k-2} + 2(q-1)^{k-2} \right) -4(q-1)^{k-2}q^{k-2} -2q^{k-2}.
\end{aligned}
\end{equation*}

If $k$ is even, then:
\begin{equation*}
\begin{aligned}
A_{SL_2}(q) &= (q^3-q)^{k-2} + (q^2-1)^{k-2} + \left( \frac{(q-1)}{2} + 2^{k-1} \right)\left( (q^2+q)^{k-2} + (q^2-q)^{k-2} \right) \\
& \; \; \;\; \; -3(q^2-q)^{k-2} - 2^k q^{k-2} +q\left( (q+1)^{k-2} + (q-1)^{k-2} \right).
\end{aligned}
\end{equation*}

If $k$ is odd, then:
\begin{equation*}
\begin{aligned}
A_{SL_2}(q) &= (q^3-q)^{k-2} + (q^2-1)^{k-2} -2^{k-1}(q^2+q)^{k-2} + (2^{k-1}-3)(q^2-q)^{k-2} \\
& \; \; \; \; \;  + q \left( (q+1)^{k-2} + (q-1)^{k-2} \right).
\end{aligned}
\end{equation*}
\end{theorem}
These results are new, except for the $SL_2(\mathbb{C})$ $E$-polynomials for $k=2,3$, which were calculated in \cite{mart} by different methods. Setting $q=1$ in these expressions gives the topological Euler characteristic of the corresponding complex character variety. Thus we obtain:
\begin{theorem}
For $k>2$, the Euler characteristic of $Rep( \pi_1(\Sigma_k) , SL_2(\mathbb{C}))$ is $2^{2k-3} -3 \cdot 2^{k-2}$ if $k$ is even, $-2^{2k-3} + 2^{k-2}$ if $k$ is odd.
\end{theorem}

\begin{theorem}[Torus knots]
Let $a,b$ be coprime integers and $K \subset S^3$ an $(a,b)$-torus knot. Let $\Gamma = \pi_1(S^3 \setminus K)$ be the fundamental group of the complement of $K$ in $S^3$.\\

If $a,b$ are odd, then:
\begin{equation*}
A_{GL_2}(q) = (q-1)\left( q + (a-1)(b-1)\frac{(q-2)}{4} \right).
\end{equation*}

If $a$ is even and $b$ is odd, then:
\begin{equation*}
A_{GL_2}(q) = (q-1)\left( q + (b-1)\frac{(aq -3a+4)}{4} \right).
\end{equation*}

In the $SL_2$ case, we have:
\begin{equation*}
A_{SL_2}(q) = q + \frac{1}{2}(a-1)(b-1)(q-2).
\end{equation*}

\end{theorem}
The $E$-polynomial for these $SL_2(\mathbb{C})$-character varieties follows from \cite{mun}, where the structure of the character variety is completely described, however the $GL_2(\mathbb{C})$ case appears to be new.

\begin{theorem}[Compact oriented surfaces, $n=3$]\label{thmcos3}
Let $\Gamma = \pi_1(\Sigma_g)$ be the fundamental group of a compact oriented surface $\Sigma_g$ of genus $g$. Then:

\begin{equation*}
\begin{aligned}
\frac{A_{GL_3}(q)}{(q-1)^{2g}} &= \frac{(q^2+q)}{3} \left( q^5+q^4+q^3 \right)^{2g-2} + \frac{(q^2-q)}{2} \left( q^5-q^3 \right)^{2g-2} \\
& \; \; \; \; \; + \left( q^8-q^6-q^5+q^3 \right)^{2g-2} + \left(  q^6-q^5-q^3+q^2 \right)^{2g-2} \\
& \; \; \; \; \; + \left( q^5-q^3-q^2+1 \right)^{2g-2} + (q-2)\left( q^6-q^5-q^4+q^3 \right)^{2g-2} \\
& \; \; \; \; \; + (q-2)\left( q^5-q^4-q^3+q^2 \right)^{2g-2} + \frac{(q-2)(q-3)}{6}\left( q^5-2q^4+q^3 \right)^{2g-2} \\
& \; \; \; \; \; + \left( q-2q^{4g-4} \right)\left(\frac{(q-2)}{2}(q^2-q)^{2g-2}(q-1)^{2g-1} + \frac{q}{2}(q^2+q)^{2g-2}(q-1)^{2g-1}\right) \\
& \; \; \; \; \; + \left( q-2q^{4g-4} \right)\left((q^3-q)^{2g-2}(q-1)^{2g-1} + (q^2-1)^{2g-2}(q-1)^{2g-1}\right) \\
& \; \; \; \; \; + \frac{1}{3}(q^2+q)(q^2+q+1)^{2g-1} +\frac{1}{2}(q^2-q)(q^2-1)^{2g-1} -q^{6g-6} \\
& \; \; \; \; \; + \left( \frac{(q^2+q)}{6} + q^{6g-6}-q^{2g-1} \right) (q-1)^{4g-2} \\
& \; \; \; \; \; +\left( q^{4g-6} - q^{2g-2} -q^{2g-4}\right)(q-1)^{2g} + (q^{2g-2}-1)(q^{2g-4}+q^{2g-1}-2) (q-1)^{2g-2}.
\end{aligned}
\end{equation*}

\begin{equation*}
\begin{aligned}
A_{SL_3}(q) &= \left( 2 \cdot 3^{2g-1} + \frac{(q-1)(q+2)}{3}\right) \left( q^5+q^4+q^3 \right)^{2g-2} + \frac{(q^2-q)}{2} \left( q^5-q^3 \right)^{2g-2} \\
& \; \; \; \; \; + \left( q^8-q^6-q^5+q^3 \right)^{2g-2} + \left(  q^6-q^5-q^3+q^2 \right)^{2g-2} \\
& \; \; \; \; \; + \left( q^5-q^3-q^2+1 \right)^{2g-2} + (q-2)\left( q^6-q^5-q^4+q^3 \right)^{2g-2} \\
& \; \; \; \; \; + (q-2)\left( q^5-q^4-q^3+q^2 \right)^{2g-2} + \left( 3^{2g-1} + \frac{(q-1)(q-4)}{6} \right) \left( q^5-2q^4+q^3 \right)^{2g-2} \\
& \; \; \; \; \; + (q-2q^{4g-4})\left( \frac{(q-2)}{2}(q^2-q)^{2g-2}(q-1)^{2g-1} + \frac{q}{2}(q^2+q)^{2g-2}(q-1)^{2g-1}\right) \\
& \; \; \; \; \; + (q-2q^{4g-4})\left( (q^3-q)^{2g-2}(q-1)^{2g-1} + (q^2-1)^{2g-2}(q-1)^{2g-1} \right) \\
& \; \; \; \; \; + \frac{1}{3}(q^2+q)(q^2+q+1)^{2g-1} +\frac{1}{2}(q^2-q)(q^2-1)^{2g-1} -q^{6g-6} 3^{2g} \\
& \; \; \; \; \; + \left( \frac{(q^2+q)}{6} + q^{6g-6}-q^{2g-1} \right) (q-1)^{4g-2} \\
& \; \; \; \; \; +\left( q^{4g-6} - q^{2g-2} -q^{2g-4}\right)(q-1)^{2g} + (q^{2g-2}-1)(q^{2g-4}+q^{2g-1}-2) (q-1)^{2g-2}.
\end{aligned}
\end{equation*}
\end{theorem}
Both of these results are new. Setting $q=1$, we obtain the Euler characteristic of the $SL_3(\mathbb{C})$-character varieties:
\begin{theorem}
For $g > 1$, the Euler characteristic of $Rep(\pi_1(\Sigma_g) , SL_3(\mathbb{C}))$ is $2 \cdot 3^{4g-3} - 7 \cdot 3^{2g-2}$.
\end{theorem}
We note that the  Euler characteristic of the $GL_2(\mathbb{C})$ and $GL_3(\mathbb{C})$-character varieties vanish due to the overall factor $(q-1)$. 

\subsection{Structure of paper}
A brief outline of the paper is as follows. In \textsection \ref{secepoly} we recall the definition of the $E$-polynomial of a complex algebraic variety. We then recall the theorem of Katz from \cite{harv} which relates the $E$-polynomial to counting points of varieties over finite fields. In \textsection \ref{secfqpts} we recall the definition of the character variety $Rep(\Gamma , G(\mathbb{C}))$ for a finitely generated group $\Gamma$ and complex reductive group $G(\mathbb{C})$. We construct a spreading out of $Rep(\Gamma , G(\mathbb{C}))$ and proceed to characterise the $\mathbb{F}_q$-points for a finite field $\mathbb{F}_q$. For $G = GL_n(\mathbb{C})$ or $SL_n(\mathbb{C})$ the answer is given by Theorem \ref{thmcountingred}, which shows that the $\mathbb{F}_q$-points are given by reductive representations of $\Gamma$ over $\mathbb{F}_q$. The remainder of the paper is concerned with counting the number of reductive representations of various groups $\Gamma$ into $GL_n(\mathbb{F}_q)$ and $SL_n(\mathbb{F}_q)$ for $n = 2,3$. In \textsection \ref{seccounting} we begin by outlining our general strategy for relating the number of reductive representations to the total number of homomorphisms $\Gamma \to GL_n(\mathbb{F}_q)$ or $SL_n(\mathbb{F}_q)$. In \ref{seccaseofgl2sl2} we carry out this strategy in the $n=2$ case for arbitrary $\Gamma$. In \ref{seccasegl3sl3} we carry out the $n=3$ case when $\Gamma$ is the fundamental group of a compact oriented surface. In \textsection \ref{seccountinghoms} we recall how character theory can be used to count the number of homomorphisms $\Gamma \to G$, where $G$ is a finite group. We work out the details of this for several different groups $\Gamma$. In \ref{sechomsgl3}, \ref{sechomssl3} we use the classification of the characters of general linear and special linear groups over finite fields to explicitly compute the number of homomorphisms $\Gamma \to GL_3(\mathbb{F}_q)$ and $SL_3(\mathbb{F}_q)$ where $\Gamma$ is the fundamental group of a compact oriented surface. In \textsection \ref{seccomputgl2sl2}, we carry out the explicit computation of the $E$-polynomial of $GL_2(\mathbb{C})$ and $SL_2(\mathbb{C})$ for various groups $\Gamma$. We conclude in \textsection \ref{seccomputgl3sl3} with the computation of the $E$-polynomial of the $GL_3(\mathbb{C})$ and $SL_3(\mathbb{C})$-character varieties for $\Gamma$ the fundamental group of a compact orientable surface.

\section{$E$-polynomials and a theorem of Katz}\label{secepoly}

Let $X$ be a complex algebraic variety. Deligne proved that the singular cohomology groups of $X$ possess a mixed Hodge structure \cite{del1,del2}. Namely, for each $i$ there is an increasing weight filtration
\begin{equation*}
0 = W_{-1} \subseteq W_0 \subseteq \dots \subseteq W_{2i} = H^i(X , \mathbb{Q} )
\end{equation*}
and a decreasing Hodge filtration\\
\begin{equation*}
H^i(X , \mathbb{C}) = F^0 \supseteq F^1 \subseteq \dots \supseteq F^m \supseteq F^{m+1} = 0
\end{equation*}
such that for each $l$, the filtration on the complexification of $Gr^W_l = W_l/W_{l-1}$ induced by $F$ is a pure Hodge structure of weight $l$. In other words, letting $Gr^{W_\mathbb{C}}_l$ denote the complexification of $Gr^W_l$ and letting $F^p Gr^{W_\mathbb{C}}_l$ denote the induced filtration on $Gr^{W_\mathbb{C}}_l$, we have that
\begin{equation*}
Gr^{W_\mathbb{C}}_l = F^p Gr^{W_\mathbb{C}}_l \oplus \overline{ F^{l-p+1} Gr^{W_\mathbb{C}}_l }
\end{equation*}
for each $0 \le p \le l$. We define the mixed Hodge numbers $h^{p,q,i}(X)$ of $X$ by:
\begin{equation*}
h^{p,q,i}(X) = dim_{\mathbb{C}} \left( F^p Gr^{W_\mathbb{C}}_i / F^{p+1} Gr^{W_\mathbb{C}}_i \right).
\end{equation*}
Similarly, the compactly supported singular cohomology groups $H^i_c(X , \mathbb{Q})$ can be shown to admit a mixed Hodge structure \cite{dakh}. Let $h^{p,q,i}_c(X)$ denote the mixed Hodge numbers for compactly supported cohomology. The {\em $E$-polynomial} or {\em Hodge-Deligne polynomial} \cite{dakh} of $X$ is the polynomial 
\begin{equation}\label{equepoly}
E_X(u,v) = \sum_{p,q,i} (-1)^i h^{p,q,i}_c(X) u^p v^q \in \mathbb{Z}[u,v].
\end{equation}
For the varieties that we are concerned with in this paper, the $E$-polynomial will turn out to be a polynomial of $uv$ alone. In this case we set $q=uv$ and denote the $E$-polynomial of $X$ by $E_X(q)$.\\

The $E$-polynomial satisfies a number of useful properties including the following:
\begin{itemize}
\item[(1)]{If $X$ is a disjoint union of finite number of locally closed subvarieties $X_i , i\in I$ then $E_X = \sum_{i \in I} E_{X_i}$ \cite[Proposition 1.6]{dakh}.}
\item[(2)]{If $X \to Y$ is a fibre bundle with fibre $F$, which is locally trivial in the Zariski topology, then $E_X = E_Y E_F$ \cite[Corollary 1.9]{dakh}.}
\item[(3)]{$E_X(1,1)$ is the topological Euler characteristic (follows from (\ref{equepoly}) and the fact that the compactly supported Euler characteristic equals the usual one, by \cite{lau}).\\}
\end{itemize}

Following Hausel and Rodriguez-Villegas \cite{harv}, we will compute the $E$-polynomial using arithmetic techniques. We recall the setup from \cite{harv}. Let $X$ be a complex algebraic variety. A {\em spreading out} of $X$ is a separated $R$-scheme of finite type $X_R$, where $R$ is a subring of $\mathbb{C}$ which is finitely generated as a $\mathbb{Z}$-algebra, such that $X = X_R \times_R \mathbb{C}$. We say that $X$ has {\em polynomial count} if there is a polynomial $P_X(t) \in \mathbb{C}[t]$ and a spreading out $X_R$ such that for every homomorphism $\phi \colon R \to \mathbb{F}_q$ to a finite field of order $q$, the number of $\mathbb{F}_q$-points of $X_R$ is $P_X(q)$. If $X$ has polynomial count it can be shown that $P_X(t) \in \mathbb{Z}[t]$ and is independent of the choice of spreading out \cite[Section 6]{harv}. The following theorem is essential to this paper.

\begin{theorem}[Katz \cite{harv}]\label{thmkatz}
Let $X$ be a complex algebraic variety. If $X$ has polynomial count with counting polynomial $P_X(t) \in \mathbb{Z}[t]$, then the $E$-polynomial $E_X(u,v)$ of $X$ is given by:
\begin{equation*}
E_X(u,v) = P_X(uv).
\end{equation*}
In particular, $E_X(u,v)$ is a polynomial of $q = uv$ alone and we write $E_X(q) = P_X(q)$.
\end{theorem}

Theorem \ref{thmkatz} gives an arithmetic technique for computing the $E$-polynomials of certain complex algebraic varieties. We will apply this to a number of singular character varieties.

\section{$\mathbb{F}_q$-points of character varieties for $SL_n$ and $GL_n$}\label{secfqpts}

\subsection{Character varieties}\label{seccv}
Let $\Gamma$ denote a finitely generated group and $G$ a reductive group scheme over $\mathbb{Z}$, given as a closed subgroup of $GL_N(\mathbb{Z})$ for some $N$. Let $\Aff_\mathbb{Z}$ denote the category of affine $\mathbb{Z}$-schemes, $Set$ the category of sets and consider the contravariant functor $F \colon \Aff_\mathbb{Z} \to Set$ sending the affine scheme $Spec(R)$ to $Hom( \Gamma , G(R))$, the set of homomorphisms from $\Gamma$ into $G(R)$, the group of $R$-points of $G$. We claim that this functor is representable by an affine $\mathbb{Z}$-scheme of finite type $X$. To show this, choose a presentation
\begin{equation*}
\Gamma = \langle x_1 , \dots , x_k \; | \;  r_j(x_1, \dots , x_k) , j \in J   \rangle
\end{equation*}
of $\Gamma$. Identify $G$ with a closed subscheme of $GL_N(\mathbb{Z})$. The matrix coefficients $(r_j)^a_b$ of the relations $r_j$ define elements of $\mathbb{Z}[G^k]$. Let $I$ be the ideal in $\mathbb{Z}[G^k]$ generated by the $\{ (r_j)^a_b - \delta^a_b \}_{j \in J, a,b = 1, \dots , N}$, where $\delta^a_b$ is the Kronecker delta and set $S = \mathbb{Z}[G^k]/I$. Then $X = Spec(S)$ is an affine $\mathbb{Z}$-scheme of finite type. Given a commutative ring $R$, the set of $R$-points of $X$ is:
\begin{equation*}
\begin{aligned}
X(R) &= Hom_{\Aff_{\mathbb{Z}}}( Spec(R) , X ) \\
&= Hom_{Ring}( S , R ) \\
&= Hom_{Ring}( \mathbb{Z}[G^k]/I , R) \\
&= Hom(\Gamma , G(R) ),
\end{aligned}
\end{equation*}
showing that $X$ does indeed represent the functor $F$. The Yoneda lemma shows that $X$ is independent of the choice of presentation of $\Gamma$ and we will write $X = Hom( \Gamma , G)$. Note that in general $X$ need not be reduced.\\

The group scheme $G$ acts on $G^k$ by conjugation componentwise. The ideal $I$ is preserved by $G$, so we obtain an action of $G$ on $X = Spec( S )$. On passing to $R$-points, this gives the action of $G(R)$ on $X(R) = Hom(\Gamma , G(R))$ by conjugation.\\

Let $X_\mathbb{C} = X \times_\mathbb{Z} \mathbb{C}$. Then $X_\mathbb{C}$ is a complex algebraic variety whose closed points are naturally in bijection with $Hom(\Gamma , G(\mathbb{C}))$. Clearly $X_\mathbb{C} = Spec( S_\mathbb{C} )$, where $S_\mathbb{C} = S \otimes_\mathbb{Z} \mathbb{C}$. The group $G(\mathbb{C})$ acts by conjugation on $X_\mathbb{C}$. We define the {\em complex character variety} $Y_\mathbb{C} = Rep( \Gamma , G(\mathbb{C}))$ to be the affine GIT quotient of this action, i.e.
\begin{equation*}
Rep( \Gamma , G(\mathbb{C}) ) = Hom(\Gamma , G(\mathbb{C}))//G(\mathbb{C}) = Spec( S_\mathbb{C}^{ G(\mathbb{C}) } ),
\end{equation*}
where $S_\mathbb{C}^{ G(\mathbb{C}) }$ denotes the subring of $G(\mathbb{C})$-invariants of $S_\mathbb{C}$.\\

Next we look for a spreading out of the character variety $Y_\mathbb{C} = Rep( \Gamma , G(\mathbb{C}) )$. Let $R \subset \mathbb{C}$ be a subring which is finitely generated as a $\mathbb{Z}$-algebra. Let $S_R = S \otimes_\mathbb{Z} R$, $X_R = X \times_\mathbb{Z} R = Spec( S_R)$ and $Y_R = Spec( S_R^{G(R)} )$. By a theorem of Seshadri \cite[Lemma 2]{ses}, since the inclusion $i \colon R \to \mathbb{C}$ is a flat morphism, we have
\begin{equation*}
S_R^{G(R)} \otimes_R \mathbb{C} = ( S_R \otimes_R \mathbb{C})^{G(\mathbb{C})} = S_\mathbb{C}^{G(\mathbb{C})}.
\end{equation*}
Thus $Y_R \times_R \mathbb{C} = Y_\mathbb{C}$, showing that $Y_R$ is a spreading out of the character variety $Y_\mathbb{C} = Rep( \Gamma , G(\mathbb{C}) )$.\\

For reasons that will fully become clear later, it will be convenient to let $R$ be the ring $R = \mathbb{Z}[1/N , \zeta_N ]$, where $N$ is some positive integer and $\zeta_N$ is a primitive $N$-th root of unity. A partial justification for this choice is given by the following lemma.

\begin{lemma}\label{lemq1}
Let $\mathbb{F}_q$ be a finite field of order $q$. There exists a (unital) homomorphism $\phi \colon R \to \mathbb{F}_q$ if and only if $q = 1 \; ({\rm mod} \; N)$.
\end{lemma}
\begin{proof}
Suppose that $\phi \colon R \to \mathbb{F}_q$ is a (unital) homomorphism. Then since $N$ is a unit of $R$ we must have that $q$ and $N$ are coprime. It follows that $N$ divides $q^d - 1$ for some $d$ and hence there is a degree $d$ extension $\mathbb{F}_{q^d} \supseteq \mathbb{F}_q$ which contains a primitive $N$-th root of unity, i.e. an element $x \in \mathbb{F}_{q^d}$ such that $x^N = 1$ but $x^a \neq 1$ for all $0 < a < N$. Let $\Phi_N(x) \in \mathbb{Z}[x]$ denote the $N$-th cyclotomic polynomial and $\Phi_N^p(x) \in \mathbb{Z}_p[x]$ the mod $p$ reduction of $\Phi_N(x)$, where $p$ is the characteristic of $\mathbb{F}_q$. The roots of $\Phi_N^p(x)$ are precisely the primitive $N$-th roots of unity in $\mathbb{F}_{q^d}$. In particular, since $\phi(\zeta_N)$ is a root of $\Phi_N^p(x)$, then $\phi(\zeta_N)$ is a primitive $N$-th root of unity. But $\phi(\zeta_N) \in \mathbb{F}_q$, so we must have that $N$ divides $q-1$. The converse is straightforward, since $R \cong \mathbb{Z}[1/N , x]/ ( \Phi_N(x) )$.
\end{proof}

\subsection{$\mathbb{F}_q$-points}\label{secfq}
Let $\phi \colon R \to \mathbb{F}_q$ be a homomorphism. This makes $\mathbb{F}_q$ into an $R$-scheme. We wish to count the number of $\mathbb{F}_q$-points of $Y_R$, i.e. the number of $R$-morphisms $Spec(\mathbb{F}_q) \to Y_R$. Noting that the inclusion $\mathbb{Z} \to R$ is flat, we have again by \cite[Lemma 2]{ses} that
\begin{equation*}
S^{G(\mathbb{Z})} \otimes_\mathbb{Z} R = S_R^{G(R)}.
\end{equation*}
Therefore,
\begin{equation*}
S_R^{G(R)} \otimes_R \mathbb{F}_q = \left( S^{G(\mathbb{Z})} \otimes_\mathbb{Z} R \right) \otimes_R \mathbb{F}_q = S^{G(\mathbb{Z})} \otimes_\mathbb{Z} \mathbb{F}_q.
\end{equation*}
So the $\mathbb{F}_q$-points of $Y_R$ may be canonically identified with the $\mathbb{F}_q$-points of $Y_\mathbb{Z}$. Note that this description does not involve the ring $R$. In fact the role of the ring $R$ is simply that, according to Lemma \ref{lemq1}, it allows us to restrict ourselves to those $q$ satisfying $q = 1 \; ({\rm mod} \; N)$ for any desired positive integer $N$.\\

To simplify notation, we will write $Y$ for $Y_\mathbb{Z}$. To count the number of $\mathbb{F}_q$-points of $Y$, it turns out to be easiest to first consider points over an algebraic completion $\overline{\mathbb{F}}_q$ of $\mathbb{F}_q$ and then count fixed points of the (arithmetic) Frobenius map. To describe the $\overline{\mathbb{F}}_q$-points of $Y$ it is useful to first consider $X$. Observe that the $\overline{\mathbb{F}}_q$-points of $X$ can be canonically identified with the set $Hom( \Gamma , G(\overline{\mathbb{F}}_q) )$ and note that $G(\overline{\mathbb{F}}_q)$ acts on $Hom( \Gamma , G(\overline{\mathbb{F}}_q) )$ by conjugation. By another theorem of Seshadri \cite[Theorem 3 (ii)]{ses}, the $\overline{\mathbb{F}}_q$-points of $Y$ can be described in terms of the $\overline{\mathbb{F}}_q$-points of $X$ as follows. Write $X(\overline{\mathbb{F}}_q)$ and $Y(\overline{\mathbb{F}}_q)$ for the sets of $\overline{\mathbb{F}}_q$-points of $X$ and $Y$. If $x \in X(\overline{\mathbb{F}}_q)$, let $\mathcal{O}_x$ denote the $G(\overline{\mathbb{F}}_q)$-orbit of $x$. Then $Y(\overline{\mathbb{F}}_q)$ is the quotient of $X(\overline{\mathbb{F}}_q)$ by the following equivalence relation:
\begin{equation}\label{equequiv}
x_1 \sim x_2 \text{ if and only if } \overline{\mathcal{O}_{x_1}} \cap \overline{\mathcal{O}_{x_2}} \neq \varnothing, \text{ where } \overline{\mathcal{O}_{x_i}} \text{ denotes the closure of } \mathcal{O}_{x_i} \text{ in } X(\overline{\mathbb{F}}_q).
\end{equation}

We will now specialise to the case where $G = GL_n(\mathbb{Z})$ or $SL_n(\mathbb{Z})$. Consider first the case $G = GL_n(\mathbb{Z})$. Then $X(\overline{\mathbb{F}}_q)$ is the set of homomorphisms $\rho \colon \Gamma \to GL_n(\overline{\mathbb{F}}_q)$. In particular, $\rho$ can be thought of as an $n$-dimensional representation of $\Gamma$ over the field $\overline{\mathbb{F}}_q$. The action of $GL_n(\overline{\mathbb{F}}_q)$ on $X(\overline{\mathbb{F}}_q)$ is by conjugation $\rho \mapsto g \rho g^{-1}$, $g \in GL_n(\overline{\mathbb{F}}_q)$. Clearly this descends to an action of $PGL_n(\overline{\mathbb{F}}_q)$. We have that two homomorphisms $\rho_1,\rho_2 \colon \Gamma \to GL_n(\overline{\mathbb{F}}_q)$ are conjugate under the $GL_n(\overline{\mathbb{F}}_q)$-action if and only if they are isomorphic as representations. Next consider the case $G = SL_n(\mathbb{Z})$. Then $X(\overline{\mathbb{F}}_q)$ is the set of homomorphisms $\rho \colon \Gamma \to SL_n(\overline{\mathbb{F}}_q)$, i.e. $n$-dimensional representations over $\overline{\mathbb{F}}_q$ with trivial determinant. Two such homomorphisms are conjugate under the $SL_n(\overline{\mathbb{F}}_q)$-action if and only if they are isomorphic as representations. This is true because the projection $SL_n(\overline{\mathbb{F}}_q) \to PGL_n(\overline{\mathbb{F}}_q)$ is surjective (as $\overline{\mathbb{F}}_q$ is algebraically closed).

\begin{proposition}\label{propfqbarpts}
For $G = GL_n(\mathbb{Z})$, the set of $\overline{\mathbb{F}}_q$-points of $Y$ may be identified with rank $n$ reductive representations of $\Gamma$ over $\overline{\mathbb{F}}_q$. For $G = SL_n(\mathbb{Z})$, the set of $\overline{\mathbb{F}}_q$-points of $Y$ may be identified with the rank $n$ reductive representations of $\Gamma$ over $\overline{\mathbb{F}}_q$ with trivial determinant.
\end{proposition}
\begin{proof}
Let $E$ be a rank $n$ representation of $\Gamma$ over $\overline{\mathbb{F}}_q$. Suppose that $E$ is an extension $0 \to A \to E \to B \to 0$, where $A$ has rank $n_1$ and $B$ has rank $n_2$. Thus $E$ is given by a homomorphism $\rho \colon \Gamma \to GL_n(\overline{\mathbb{F}}_q)$ of the form
\begin{equation*}
\rho(x) = \left[ \begin{matrix} \rho_A(x) & \psi(x) \\ 0 & \rho_B(x) \end{matrix} \right],
\end{equation*}
where $\rho_A \colon \Gamma \to GL_{n_1}(\overline{\mathbb{F}}_q)$, $\rho_B\colon \Gamma \to GL_{n_2}(\overline{\mathbb{F}}_q)$ are homomorphisms corresponding to the representations $A$ and $B$. For any $t \in \overline{\mathbb{F}}_q^\times$, let $D_t \in SL_n(\overline{\mathbb{F}}_q)$ be given by:
\begin{equation*}
D_t = \left[ \begin{matrix} t^{n_2} & 0 \\ 0 & t^{-n_1} \end{matrix} \right].
\end{equation*}
Then
\begin{equation*}
D_t \rho(x) D_t^{-1} = \left[ \begin{matrix} \rho_A(x) & t^{n} \psi(x) \\ 0 & \rho_B(x) \end{matrix} \right].
\end{equation*}
In particular, the closure of the orbit $\mathcal{O}_\rho$ contains the representation $\rho_A \oplus \rho_B$. Thus $\rho \sim \rho_A \oplus \rho_B$. More generally, if a representation $E$ admits a filtration $0 = F_0 \subset F_1 \subset F_2 \subset \dots \subset F_k = E$, then $E$ is equivalent to the associated graded representation $(F_1/F_0) \oplus (F_2/F_1) \oplus \dots \oplus (F_k/F_{k-1})$. Now since every representation of $\Gamma$ admits a Jordan-H\"older decomposition, we see that every representation is equivalent to a reductive representation. It remains to show that non-isomorphic reductive representations are not equivalent.\\

Let $G(\overline{\mathbb{F}}_q)$ be either $GL_n(\overline{\mathbb{F}}_q)$ or $SL_n(\overline{\mathbb{F}}_q)$ and let $tr \colon G(\overline{\mathbb{F}}_q) \to \overline{\mathbb{F}}_q$ be the trace. Given $x \in \Gamma$, consider the regular function $f_x \colon Hom(\Gamma , G(\overline{\mathbb{F}}_q) ) \to \overline{\mathbb{F}}_q$ given by $f_x(\rho) = tr( \rho(x))$. These are conjugation invariant functions and hence define regular functions on $Y_{\overline{\mathbb{F}}_q} = Y \times_\mathbb{Z} \overline{\mathbb{F}}_q$. The characters of distinct irreducible representations of $\Gamma$ are linearly independent \cite[Theorem 3.6.2(i)]{eghlsvy}, thus if $\rho_1,\rho_2$ are non-isomorphic reductive representations, then there is some $x \in \Gamma$ for which $f_x(\rho_1) \neq f_x(\rho_2)$. It follows that $\overline{\mathcal{O}_{\rho_1}} \cap \overline{\mathcal{O}_{\rho_2}} = \varnothing$, since $f_x$ is constant on $\mathcal{O}_{\rho_1}$ and $\mathcal{O}_{\rho_2}$ with different values.

\end{proof}

Now we want to pass from the $\overline{\mathbb{F}}_q$-points of $Y$ to the $\mathbb{F}_q$-points. Let $\sigma \colon \overline{\mathbb{F}}_q \to \overline{\mathbb{F}}_q$ be the Frobenius automorphism $\sigma(x) = x^q$. The action of $\sigma$ on $\overline{\mathbb{F}}_q$ induces a corresponding action on $Y(\overline{\mathbb{F}}_q)$. Namely, if $f \colon Spec(\overline{\mathbb{F}}_q) \to Y$ is an $\overline{\mathbb{F}}_q$-point of $Y$, then we get a new $\overline{\mathbb{F}}_q$-point by taking the composition $f \circ \sigma^* \colon Spec(\overline{\mathbb{F}}_q) \to Y$. By abuse of notation we will denote this action as $\sigma \colon Y(\overline{\mathbb{F}}_q) \to Y(\overline{\mathbb{F}}_q)$. This is usually called the arithmetic Frobenius map. From the definition it is immediate that the fixed point set of $\sigma \colon Y(\overline{\mathbb{F}}_q) \to Y(\overline{\mathbb{F}}_q)$ is precisely the set $Y(\mathbb{F}_q)$ of $\mathbb{F}_q$-points of $Y$. Thus we will count $Y(\mathbb{F}_q)$ by counting the fixed point set of $\sigma$ on $Y(\overline{\mathbb{F}}_q)$.\\

We let the Frobenius map $\sigma$ act on $GL_n(\overline{\mathbb{F}}_q)$ and $SL_n(\overline{\mathbb{F}}_q)$ by sending a matrix $A$ with entries $A_{ij} \in \overline{\mathbb{F}}_q$ to the matrix $\sigma(A)$ with entries $\sigma(A)_{ij} = \sigma(A_{ij})$. Recall that $Y(\overline{\mathbb{F}}_q)$ can be canonically identified with the set of isomorphism classes of reductive representations $\rho \colon \Gamma \to G(\overline{\mathbb{F}}_q)$, where $G(\overline{\mathbb{F}}_q)$ is either $GL_n(\overline{\mathbb{F}}_q)$ or $SL_n(\overline{\mathbb{F}}_q)$. Then it is easy to see that the Frobenius map $\sigma \colon Y(\overline{\mathbb{F}}_q) \to Y(\overline{\mathbb{F}}_q)$ sends a representation $\rho$ to the composition $\sigma \circ \rho$.

\begin{lemma}\label{lemfixrep}
Let $\rho \colon \Gamma \to GL_n(\overline{\mathbb{F}}_q)$ be a reductive representation such that $\sigma \circ \rho$ is conjugate to $\rho$. Then $\rho$ is conjugate to a representation $\rho' \colon \Gamma \to GL_n(\mathbb{F}_q)$.
\end{lemma}
\begin{proof}
First we prove the result for irreducible representations. Thus suppose that $\rho \colon \Gamma \to GL_n(\overline{\mathbb{F}}_q)$ is irreducible and $\sigma \circ \rho$ is conjugate to $\rho$, that is, 
\begin{equation}\label{equsigmacong}
\sigma \circ \rho = g^{-1} \rho g,
\end{equation}
for some $g \in GL_n(\overline{\mathbb{F}}_q)$. Note that for some positive integer $d$ we have that $\rho$ is valued in $GL_n(\mathbb{F}_{q^d})$. Indeed, if $\Gamma$ is generated by $x_1, x_2, \dots x_k$, then the matrix coefficients of $\rho(x_1), \rho(x_2), \dots , \rho(x_k)$ are contained in some finite extension $\mathbb{F}_{q^d}$ of $\mathbb{F}_q$. Thus $\sigma^d(\rho) = \rho$. Combined with (\ref{equsigmacong}), this gives $\rho = \sigma^d(\rho) = u^{-1} \rho u$, where $u = g \sigma(g) \sigma^2(g) \dots \sigma^{d-1}(g)$. Then since $\rho$ is irreducible, we have that $u$ is a multiple of the identity, that is
\begin{equation*}
g \sigma(g) \sigma^2(g) \dots \sigma^{d-1}(g) = c Id,
\end{equation*}
for some $c \in \overline{\mathbb{F}}_q^\times$. Since $\overline{\mathbb{F}}_q$ is algebraically closed we can find an $a \in \overline{\mathbb{F}}_q^\times$ such that $a \sigma(a) \dots \sigma^{d-1}(a) = a^{1+q + \dots + q^{d-1}} = c^{-1}$. Replacing $g$ by $ag$, we may assume that 
\begin{equation*}
g \sigma(g) \sigma^2(g) \dots \sigma^{d-1}(g) = Id.
\end{equation*}
From this it follows easily that $\sigma^d(g) = g$, i.e. $g \in GL_n(\mathbb{F}_{q^d})$. Applying Hilbert's Theorem 90, we have that there exists $h \in GL_n(\mathbb{F}_{q^d})$ such that $g = h^{-1}\sigma(h)$. It follows that $\sigma( h\rho h^{-1}) = h \rho h^{-1}$, hence $h \rho h^{-1}$ is valued in $GL_n(\mathbb{F}_q)$, as required.\\

Now we consider the general case where $\rho$ is reductive. Let $\chi_\rho \colon \Gamma \to \overline{\mathbb{F}}_q$ be the character of $\rho$, that is $\chi_\rho(x) = tr( \rho(x) )$. Write $\chi_\rho = \sum_i m_i\chi_i$, where the $\chi_i$ are characters of distinct irreducible representations and $m_i$ are the multiplicities. Since $\sigma \circ \rho$ is conjugate to $\rho$, we have that $\chi_{\sigma \circ \rho} = \sigma( \chi_\rho) = \chi_\rho$. For a given irreducible representation $\rho_j$ with character $\chi_j$, let $d_j$ be the order of $\chi_j$ under the Frobenius action, i.e. let $d_j$ be the least positive integer such that $\sigma^{d_j}(\chi_j) = \chi_j$. Such a $d_j$ exists since, as we have previously argued, $\rho_j$ is valued in some finite extension of $\mathbb{F}_q$. Then since $\sigma(\chi) = \chi$, we can choose a decomposition of $\rho$ of the form 
\begin{equation*}
\rho = \bigoplus_j \left( \rho_j \oplus \sigma(\rho_j) \oplus \dots \oplus \sigma^{d_j-1}(\rho_j) \right)
\end{equation*}
where the $\rho_j$ are irreducible representations (not necessarily distinct) and $d_j$ is the order of $\chi_j$ under the Frobenius action. Set $\rho'_j = \rho_j \oplus \sigma(\rho_j) \oplus \dots \oplus \sigma^{d_j-1}(\rho_j)$ and let $\chi'_j$ be the character of $\rho'_j$. Then $\sigma(\rho'_j) = \rho'_j$, hence $\sigma \circ \chi'_j$ is conjugate to $\chi'_j$. If we can prove the lemma for the representations $\rho'_j$, then by taking direct sum we obtain the result for $\rho$. Thus we are reduced to proving the lemma for a representation of the form $\rho' = \rho \oplus \sigma(\rho) \oplus \dots \oplus \sigma^{d-1}(\rho)$, where $\rho$ is irreducible and the Frobenius action on $\chi$ has order $d$. Note this means that $\sigma^d(\rho)$ is conjugate to $\rho$. Now since $\rho$ is irreducible, the above proof of the lemma in the irreducible case implies that, upon conjugating if necessary, we can assume $\rho$ is valued in $\mathbb{F}_{q^d}$, i.e. it is a homomorphism $\rho \colon \Gamma \to GL_n(\mathbb{F}_{q^d})$. Thus we can assume $\sigma^d(\rho) = \rho$. It follows that $\sigma \circ \rho' = g^{-1} \rho' g$, where $g$ is given by:
\begin{equation*}
g = \left[\begin{matrix} 0 & & & & & & 1 \\
1 & 0 & & & & & \\
& 1 & 0 & & & & \\
& & 1 & 0 & & & \\
& & & & \ddots & & \\
& & & & & 1 & 0
\end{matrix}\right].
\end{equation*}
Clearly $g \sigma(g) \dots \sigma^{d-1}(g) = g^d = Id$. Once again, Hilbert's Theorem 90 implies that $g = h^{-1}\sigma(h)$ for some $h \in GL_n(\mathbb{F}_{q^d})$. Then $\sigma( h\rho' h^{-1}) = h\rho' h^{-1}$, so that $h \rho' h^{-1}$ is valued in $GL_n(\mathbb{F}_q)$, as required.
\end{proof}

\begin{remark}\label{remirrep}
In the proof of Lemma \ref{lemfixrep} we have shown the following. Let $V$ be an irreducible representation of $\Gamma$ over $\overline{\mathbb{F}}_q$. Let $d$ be the least positive integer such that $\sigma^d(V) \cong V$. Set $V' = V \oplus \sigma(V) \oplus \dots \oplus \sigma^{d-1}(V)$. Then $V'$ is a representation of $\Gamma$ over $\mathbb{F}_q$ and is moreover irreducible over $\mathbb{F}_q$. It is easy to see that every irreducible representation of $\Gamma$ over $\mathbb{F}_q$ is of this form for some $V$.
\end{remark}

\begin{proposition}
For $G = GL_n(\mathbb{Z})$, the set of $\mathbb{F}_q$-points of $Y$ is the set of equivalence classes of reductive representations $\rho \colon \Gamma \to GL_n(\mathbb{F}_q)$, where $\rho_1, \rho_2$ are considered equivalent if $\rho_2 = g \rho_1 g^{-1}$, for some $g \in GL_n(\overline{\mathbb{F}}_q)$. For $G = SL_n(\mathbb{Z})$, the set of $\mathbb{F}_q$-points of $Y$ is the set of equivalence classes of reductive representations $\rho \colon \Gamma \to SL_n(\mathbb{F}_q)$, where $\rho_1, \rho_2$ are considered equivalent if $\rho_2 = g \rho_1 g^{-1}$, for some $g \in SL_n(\overline{\mathbb{F}}_q)$.
\end{proposition}
\begin{proof}
According to Proposition \ref{propfqbarpts}, the $\overline{\mathbb{F}}_q$-points of $Y$ are given by isomorphism classes of reductive representations (with trivial determinant in the special linear case). The $\mathbb{F}_q$-points of $Y$ are then the fixed points of the Frobenius action. These are precisely the reductive representations $\rho$ such that $\sigma \circ \rho$ is conjugate to $\rho$. According to Lemma \ref{lemfixrep}, these are up to conjugacy, given by representations valued in $GL_n(\mathbb{F}_q)$ (in the special linear case, the determinant of the representation is trivial, so the fixed points of the Frobenius action are given by representations valued in $SL_n(\mathbb{F}_q)$). Lastly we just need to note that two such representations define the same $\mathbb{F}_q$-point of $Y$ if and only if they are conjugate over $\overline{\mathbb{F}}_q$.
\end{proof}

Next, we want to understand better what it means for two representations to be conjugate over $\overline{\mathbb{F}}_q$. For this we introduce some definitions.
\begin{definition}
Let $\rho_1, \rho_2 \colon \Gamma \to GL_n(\mathbb{F}_q)$ be two representations. We will say that $\rho_1, \rho_2$ are:
\begin{itemize}
\item[(1)]{{\em conjugate} if $\rho_2 = g \rho_1 g^{-1}$ for some $g \in GL_n(\mathbb{F}_q)$ }
\item[(2)]{{\em conjugate over $\overline{\mathbb{F}}_q$} if $\rho_2 = g \rho_1 g^{-1}$ for some $g \in GL_n(\overline{\mathbb{F}}_q)$}
\item[(3)]{{\em special conjugate over $\overline{\mathbb{F}}_q$} if $\rho_2 = g \rho_1 g^{-1}$ for some $g \in SL_n(\overline{\mathbb{F}}_q)$.}
\end{itemize}
\end{definition}

\begin{proposition}\label{propcompare}
Let $\rho_1, \rho_2 \colon \Gamma \to GL_n(\mathbb{F}_q)$ be two reductive representations. The following are equivalent:
\begin{itemize}
\item[(1)]{$\rho_1,\rho_2$ are conjugate}
\item[(2)]{$\rho_1,\rho_2$ are conjugate over $\overline{\mathbb{F}}_q$}
\item[(3)]{$\rho_1,\rho_2$ are special conjugate over $\overline{\mathbb{F}}_q$.}
\end{itemize}
\end{proposition}
\begin{proof}
(2) and (3) are equivalent because the natural map $SL_n(\overline{\mathbb{F}}_q) \to PGL_n(\overline{\mathbb{F}}_q)$ is surjective. Clearly (1) implies (2), so it remains to show that (2) implies (1). Let $\rho_1, \rho_2 \colon \Gamma \to GL_n(\mathbb{F}_q)$ and let $V,W$ be the corresponding representations over $\mathbb{F}_q$. Then $\rho_1,\rho_2$ are conjugate over $\overline{\mathbb{F}}_q$ if and only if there is an isomorphism $V \otimes_{\mathbb{F}_q} \overline{\mathbb{F}}_q \cong W \otimes_{\mathbb{F}_q} \overline{\mathbb{F}}_q$ of representations over $\overline{\mathbb{F}}_q$. Suppose this is the case. By Remark \ref{remirrep}, there are irreducible representations $V_1,V_2, \dots, V_k$ of $\Gamma$ over $\overline{\mathbb{F}}_q$ (not necessarily distinct) such that:
\begin{equation*}
V \otimes_{\mathbb{F}_q} \overline{\mathbb{F}}_q \cong \bigoplus_{i=1}^k V'_i,
\end{equation*}
where $V'_i = V_i \oplus \sigma(V_i) \oplus \dots \oplus \sigma^{d_i-1}(V_i)$ and $d_i$ is the least positive integer such that $\sigma^{d_i}(V_i) \cong V_i$. But if $V \otimes_{\mathbb{F}_q} \overline{\mathbb{F}}_q \cong W \otimes_{\mathbb{F}_q} \overline{\mathbb{F}}_q$, then the same summands $V'_i$ occur in $V$ and $W$ the same number of times, hence $V$ and $W$ are isomorphic over $\mathbb{F}_q$, i.e. $\rho_1,\rho_2$ are conjugate.
\end{proof}

By this proposition, two representations $\rho_1, \rho_2 \colon \Gamma \to GL_n(\mathbb{F}_q)$ are conjugate over $\overline{\mathbb{F}}_q$ if and only if they are conjugate over $\mathbb{F}_q$, i.e. if and only if they are isomorphic as representations over $\mathbb{F}_q$. Similarly, $\rho_1$ and $\rho_2$ are special conjugate over $\overline{\mathbb{F}}_q$ if and only if they are conjugate over $\mathbb{F}_q$, i.e. if and only if they are isomorphic as representations over $\mathbb{F}_q$. Putting together the results of this section with Katz's theorem, we have:
\begin{theorem}\label{thmcountingred}
Suppose there is a polynomial $A(t) \in \mathbb{\mathbb{C}}[t]$ and a positive integer $N$ such that for every finite field $\mathbb{F}_q$ of order $q$ with $q = 1 \; ({\rm mod} \; N)$, the number of isomorphism classes of $n$-dimensional reductive representations over $\mathbb{F}_q$ (resp. $n$-dimensional reductive representations over $\mathbb{F}_q$ with trivial determinant) equals $A(q)$. Then $A(q)$ is the $E$-polynomial of the complex character variety $Rep( \Gamma , GL_n(\mathbb{C}))$ (resp. $Rep(\Gamma , SL_n(\mathbb{C}))$). 
\end{theorem}

\section{Counting representations}\label{seccounting}

Our goal is to compute for $n=2,3$, the number of isomorphism classes of reductive representations of $\Gamma$ into $GL_n(\mathbb{F}_q)$ or $SL_n(\mathbb{F}_q)$, where by isomorphism class, we mean conjugate under the action of $GL_n(\mathbb{F}_q)$. This can be worked out from the number of absolutely irreducible representations, so our first task is to compute this. We will indirectly compute the number of absolutely irreducible representations by first computing the total number of representations and then subtracting away all representations which are not absolutely irreducible.\\

To explain our strategy further, we introduce some notation. Let $\mathcal{X}_n$ denote the set of all isomorphism classes of representations of $\Gamma$ into $GL_n(\mathbb{F}_q)$ and $\widetilde{\mathcal{X}}_n \subseteq \mathcal{X}_n$ the set of such representations with trivial determinant. Let $\mathcal{X}_n^{\rm red}$ (resp. $\widetilde{\mathcal{X}}_n^{\rm red}$) denote the subset of $\mathcal{X}_n$ (resp. $\widetilde{\mathcal{X}}_n$) consisting of reductive representations, similarly write $\mathcal{X}_n^{\rm nr}, \widetilde{\mathcal{X}}_n^{\rm nr}$ for non-reductive representations and $\mathcal{X}_n^{\rm ai}, \widetilde{\mathcal{X}}_n^{\rm ai}$ for absolutely irreducible representations. We will stratify $\mathcal{X}_n$ into a disjoint union of subsets $\mathcal{X}_n^{(i)}$ in such a way that: 
\begin{enumerate}
\item[(1)] each $\mathcal{X}_n^{(i)}$ will consist entirely of reductive representations or entirely of non-reductive representations, and \item[(2)] for every $\rho \in \mathcal{X}_n^{(i)}$, the size of the stabiliser of $\rho$ under the $PGL_n(\mathbb{F}_q)$-action by conjugation depends only on $i$ and will be denoted by $s_i$. 
\end{enumerate}
We get an induced stratification of $\widetilde{\mathcal{X}}_n$ by setting $\widetilde{\mathcal{X}}_n^{(i)} = \mathcal{X}_n^{(i)} \cap \widetilde{\mathcal{X}}_n$.\\

Let $a$ be the total number of strata, so $i$ runs from $1$ to $a$. Set $I = \{ 1 , 2 , \dots , a \}$. Let $R \subseteq I$ be the set of indices $i \in I$ for which $\mathcal{X}_n^{(i)} \subseteq \mathcal{X}_n^{\rm red}$ and $N \subseteq I$ the indices $i \in I$ with $\mathcal{X}_n^{(i)} \subseteq \mathcal{X}_n^{\rm nr}$, so $I$ is the disjoint union of $R$ and $N$. We will choose the strata in such a way that $\mathcal{X}_n^{(1)} = \mathcal{X}_n^{\rm ai}$, hence $s_1 = 1$ and $1 \in R$. We write $A_{GL_n}(q) = |\mathcal{X}_n^{\rm red}|$ and $A_{SL_n}(q) = |\widetilde{\mathcal{X}}_n^{\rm red}|$ for the number of reductive representations into $GL_n(\mathbb{F}_q)$ and $SL_n(\mathbb{F}_q)$, so that:
\begin{equation}\label{equred}
\begin{aligned}
A_{GL_n}(q) &= \sum_{ i \in R } | \mathcal{X}^{(i)}_n| = |\mathcal{X}^{\rm ai}_n| +\sum_{\substack{ i \in R \\ i \neq 1}} | \mathcal{X}^{(i)}_n| , \\
A_{SL_n}(q) &= \sum_{ i \in R } | \widetilde{\mathcal{X}}^{(i)}_n| = |\widetilde{\mathcal{X}}^{\rm ai}_n| +\sum_{\substack{ i \in R \\ i \neq 1}} | \widetilde{\mathcal{X}}^{(i)}_n|.
\end{aligned}
\end{equation}
On the other hand, the orbit-stabiliser theorem gives:
\begin{equation*}
\begin{aligned}
\frac{ | Hom(\Gamma , GL_n(\mathbb{F}_q)) |}{| PGL_n(\mathbb{F}_q)|} &= \sum_{i \in R} \frac{ |\mathcal{X}^{(i)}_n|}{s_i} + \sum_{i \in N} \frac{ |\mathcal{X}^{(i)}_n|}{s_i} =  |\mathcal{X}^{\rm ai}_n| + \sum_{\substack{ i \in R \\ i \neq 1}} \frac{ |\mathcal{X}^{(i)}_n|}{s_i} + \sum_{i \in N} \frac{ |\mathcal{X}^{(i)}_n|}{s_i}, \\
\frac{ | Hom(\Gamma , SL_n(\mathbb{F}_q)) |}{| PGL_n(\mathbb{F}_q)|} &= \sum_{i \in R} \frac{ |\widetilde{\mathcal{X}}^{(i)}_n|}{s_i} + \sum_{i \in N} \frac{ |\widetilde{\mathcal{X}}^{(i)}_n|}{s_i} = |\widetilde{\mathcal{X}}^{\rm ai}_n| + \sum_{\substack{ i \in R \\ i \neq 1}} \frac{ |\widetilde{\mathcal{X}}^{(i)}_n|}{s_i} + \sum_{i \in N} \frac{ |\widetilde{\mathcal{X}}^{(i)}_n|}{s_i}.
\end{aligned}
\end{equation*}
Re-arranging gives:
\begin{equation}\label{equai}
\begin{aligned}
|\mathcal{X}^{\rm ai}_n| &= \frac{ | Hom(\Gamma , GL_n(\mathbb{F}_q)) |}{| PGL_n(\mathbb{F}_q)|} - \sum_{\substack{ i \in R \\ i \neq 1}} \frac{ |\mathcal{X}^{(i)}_n|}{s_i} - \sum_{i \in N} \frac{ |\mathcal{X}^{(i)}_n|}{s_i}, \\
|\widetilde{\mathcal{X}}^{\rm ai}_n| &= \frac{ | Hom(\Gamma , SL_n(\mathbb{F}_q)) |}{| PGL_n(\mathbb{F}_q)|} -  \sum_{\substack{ i \in R \\ i \neq 1}} \frac{ |\widetilde{\mathcal{X}}^{(i)}_n|}{s_i} - \sum_{i \in N} \frac{ |\widetilde{\mathcal{X}}^{(i)}_n|}{s_i}.
\end{aligned}
\end{equation}
These expression will be useful as they give the number of absolutely irreducible representations in terms of the total number of homomorphisms of $\Gamma$ into $GL_n(\mathbb{F}_q)$ or $SL_n(\mathbb{F}_q)$ minus the remaining strata, which can be expressed in terms of lower rank representations. Substituting (\ref{equai}) into (\ref{equred}), we have:
\begin{equation}\label{equcounting}
\begin{aligned}
A_{GL_n}(q) &=  \frac{ | Hom(\Gamma , GL_n(\mathbb{F}_q)) |}{| PGL_n(\mathbb{F}_q)|} +\sum_{\substack{ i \in R \\ i \neq 1}} \left( 1 - \frac{1}{s_i} \right)
|\mathcal{X}^{(i)}_n| - \sum_{i \in N} \frac{1}{s_i}|\mathcal{X}^{(i)}_n|, \\
A_{SL_n}(q) &= \frac{ | Hom(\Gamma , SL_n(\mathbb{F}_q)) |}{| PGL_n(\mathbb{F}_q)|} +\sum_{\substack{ i \in R \\ i \neq 1}} \left( 1 - \frac{1}{s_i} \right) |\widetilde{\mathcal{X}}^{(i)}_n| - \sum_{i \in N} \frac{1}{s_i}|\widetilde{\mathcal{X}}^{(i)}_n|.
\end{aligned}
\end{equation}
The important feature of these expressions is that the absolutely irreducible representations do not appear on the right hand side.

\subsection{Case of $GL_2$ and $SL_2$}\label{seccaseofgl2sl2}

We restrict ourselves to odd $q$. For rank $2$ representations our stratification will consist of $6$ strata: $R = \{1,2,3,4\}, N = \{5,6\}$ as we describe below. We will also need the following notation:
\begin{itemize}
\item{For any integer $j \ge 1$, let $m_j$ be the number of homomorphisms $\Gamma \to \mathbb{Z}_j$.}
\item{Let $A$ be a $1$-dimensional representation of $\Gamma$ over $\mathbb{F}_q$. We let $b^j_A$ denote the dimension of the group cohomology $H^j(\Gamma , A)$ over $\mathbb{F}_q$. When $A$ is the trivial representation, we write $b^j$ in place of $b^j_A$.}
\item{For any integer $m \ge 0$, let $[m]_q = (q^m-1)/(q-1)$.}
\end{itemize}

{\bf Strata $\mathcal{X}_2^{(1)}, \widetilde{\mathcal{X}}_2^{(1)}$.} These are the absolutely irreducible representations.\\

{\bf Strata $\mathcal{X}_2^{(2)}, \widetilde{\mathcal{X}}_2^{(2)}$.} These are the irreducible, but not absolutely irreducible representations. Thus they have the form $A \oplus \sigma(A)$, where $A$ corresponds to a homomorphism $\rho_A : \Gamma \to \mathbb{F}_{q^2}^\times$ and $\sigma(A) \ncong A$. Thus $|\mathcal{X}_2^{(2)}| = (m_{q^2-1} - m_{q-1})/2$. For $\widetilde{\mathcal{X}}_2^{(2)}$, note that the trivial determinant condition gives $A \sigma(A) = 1$, hence $\rho_A$ is valued in $\mu_{q+1} = \{ x \in \mathbb{F}_{q^2}^\times \; | \; x^{q+1} = 1 \}$. Then since $\mu_{q+1}$ is cyclic of order $q+1$ and $\mu_{q+1} \cap \mathbb{F}_q^\times = \{ \pm 1\}$, we find that $|\widetilde{\mathcal{X}}_2^{(2)}| = (m_{q+1}-m_2)/2$. The stabiliser of $A \oplus \sigma(A)$ in $GL_2(\mathbb{F}_q)$ consists of diagonal matrices of the form $diag(x , \sigma(x))$, where $x \in \mathbb{F}_{q^2}^\times$. Hence $s_2 = q+1$.\\

{\bf Strata $\mathcal{X}_2^{(3)}, \widetilde{\mathcal{X}}_2^{(3)}$.} These are of the form $A \oplus B$, where $A,B$ are distinct rank $1$ representations over $\mathbb{F}_q$. Hence $|\mathcal{X}_2^{(3)}| = m_{q-1}(m_{q-1}-1)/2$, $|\widetilde{\mathcal{X}}_2^{(3)}| = (m_{q-1}-m_2)/2$, $s_3 = q-1$.\\

{\bf Strata $\mathcal{X}_2^{(4)}, \widetilde{\mathcal{X}}_2^{(4)}$.} These are of the form $A \oplus A$, where $A$ is a rank $1$ representation over $\mathbb{F}_q$. Hence $|\mathcal{X}_2^{(4)}| = m_{q-1}$, $|\widetilde{\mathcal{X}}_2^{(4)}| = m_2$, $s_4 = |PGL_2(\mathbb{F}_q)| = q^3-q$.\\

{\bf Strata $\mathcal{X}_2^{(5)}, \widetilde{\mathcal{X}}_2^{(5)}$.} These are non-trivial extensions $A \to E \to B $, where $A,B$ are distinct rank $1$ representations over $\mathbb{F}_q$. For fixed $A,B$, such representations correspond to elements in the projectivisation of $H^1(\Gamma , B^* \otimes A )$. Noting that $A \ncong B$ if and only if $B^* \otimes A \ncong 1$, we get $|\mathcal{X}_2^{(5)}| = m_{q-1} \sum_{\{A \neq 1\}} [b^1_A]_q$, where the sum is over the non-trivial rank $1$ representations. Similarly $|\widetilde{\mathcal{X}}_2^{(5)}| = \sum_{ \{ A | A^2 \neq 1 \} } [b^1_{A^2}]_q$, where the sum is over the non-trivial rank $1$ representations $A$ with $A^2 \ncong 1$. We also find that $s_5 = 1$.\\

{\bf Strata $\mathcal{X}_2^{(6)}, \widetilde{\mathcal{X}}_2^{(6)}$.} These are non-trivial extensions $A \to E \to A$, where $A$ is a rank $1$ representation over $\mathbb{F}_q$. Using similar reasoning as above we get $|\mathcal{X}_2^{(6)}| = m_{q-1}[b^1]_q$, $|\widetilde{\mathcal{X}}_2^{(6)}| = m_2 [b^1]_q$, $s_6 = q$.\\

\begin{table}
\begin{equation*}
\renewcommand{\arraystretch}{1.4}
\begin{tabular}{|c|c|c|c|}
\hline
$i$ & $|\mathcal{X}_2^{(i)}|$ & $|\widetilde{\mathcal{X}}_2^{(i)}|$ & $s_i$ \\
\hline
$2$ & $\frac{1}{2}(m_{q^2-1} - m_{q-1})$ & $\frac{1}{2}(m_{q+1}-m_2)$ & $(q+1)$ \\
$3$ & $\frac{1}{2}m_{q-1}(m_{q-1}-1)$ & $\frac{1}{2}(m_{q-1}-m_2)$ & $(q-1)$ \\
$4$ & $m_{q-1}$ & $m_2$ & $(q^3-q)$ \\
$5$ & $m_{q-1} \sum_{\{A \neq 1\}} [b^1_A]_q$ & $\sum_{ \{ A | A^2 \neq 1 \} } [b^1_{A^2}]_q$ & $1$ \\
$6$ & $m_{q-1}[b^1]_q$ & $m_2 [b^1]_q$ & $q$ \\
\hline
\end{tabular}
\end{equation*}
\caption{Sizes of strata and their stabilisers in the rank $2$ case.}\label{figrank2}
\end{table}

Our calculations are summarised in Table \ref{figrank2}. Putting all of this into Equation (\ref{equcounting}) and simplifying gives the following theorem.

\begin{theorem}\label{thmcountingredrank2}
Let $q$ be odd. Then $A_{GL_2}(q)$, $A_{SL_2}(q)$ are given by:
\begin{equation*}
\begin{aligned}
A_{GL_2}(q) &= \frac{ | Hom(\Gamma , GL_2(\mathbb{F}_q)) |}{| PGL_2(\mathbb{F}_q)|} + \left( 1 - \frac{1}{q+1} \right) \frac{m_{q^2-1}}{2} + \left( 1 - \frac{1}{q-1} \right) \frac{m^2_{q-1}}{2} \\
& \; \; \; \; \;  -m_{q-1}[ b^1 - 1]_q - m_{q-1} \sum_{ \{ A | A \neq 1 \} } [b^1_A]_q.
\end{aligned}
\end{equation*}

\begin{equation*}
\begin{aligned}
A_{SL_2}(q) &= \frac{ | Hom(\Gamma , SL_2(\mathbb{F}_q)) |}{| SL_2(\mathbb{F}_q)|} + \left( 1 - \frac{1}{q+1} \right) \frac{m_{q+1}}{2} + \left( 1 - \frac{1}{q-1} \right) \frac{m_{q-1}}{2} \\
& \; \; \; \; \;  -m_{2}[ b^1 - 1]_q - \sum_{ \{ A | A^2 \neq 1 \} } [b^1_{A^2}]_q.
\end{aligned}
\end{equation*}
\end{theorem}

\subsection{Case of $GL_3$ and $SL_3$}\label{seccasegl3sl3}
For rank $3$ representations, the calculation of the sizes of the strata is already quite involved and the details depend considerably on the particular choice of group $\Gamma$. As such we will restrict ourselves to the case that $\Gamma = \pi_1(\Sigma_g)$ is the fundamental group of a compact oriented surface of genus $g$. We will also assume that $q = 1 \; ({\rm mod} \; 6)$. Our chosen stratification will consist of $30$ strata: $R =\{ 1,2 , \dots , 7\}, N = \{8,9, \dots , 30\}$ described below. We will also make use of the following lemma.
\begin{lemma}\label{lemcountsl3}
Let $X = \{ (A,B,C) \in Hom(\Gamma , \mathbb{F}_q^\times) \; | \; ABC = 1 \text{ and } $A,B,C$ \text{ are distinct } \}$. Then $|X| = m_{q-1}^2 -3m_{q-1} +2m_3$.
\end{lemma}
\begin{proof}
Let $Y = \{ (A,B,C) \in Hom(\Gamma , \mathbb{F}_q^\times) \; | \; ABC = 1 \; \}$, $Y_{AB} = \{ (A,B,C) \in Y \; | \; A=B \; \}$, $Y_{BC} = \{ (A,B,C) \in Y \; | \; B=C \; \}$, $Y_{AC} = \{ (A,B,C) \in Y \; | \; A=C \; \}$, $Y_{ABC} = \{ (A,B,C) \in Y \; | \; A=B=C \; \}$. Note that $Y_{AB} \cap Y_{BC} = Y_{AB} \cap Y_{AC} = Y_{BC} \cap Y_{AC} = Y_{ABC}$. It is easy to see that $|Y| = m_{q-1}^2$, $|Y_{AB}| = |Y_{BC}| = |Y_{AC}| = m_{q-1}$ and $|Y_{ABC}| = m_3$. Thus
\begin{equation*}
|X| = |Y| - |Y_{AB}| - |Y_{BC}| - |Y_{AC}| + 2|Y_{ABC}| = m_{q-1}^2 - 3m_{q-1} + 2m_3.
\end{equation*}
\end{proof}

{\bf Strata $\mathcal{X}_3^{(1)}, \widetilde{\mathcal{X}}_3^{(1)}$.} The absolutely irreducible representations.\\

{\bf Strata $\mathcal{X}_3^{(2)}, \widetilde{\mathcal{X}}_3^{(2)}$.} The irreducible, but not absolutely irreducible representations. These have the form $A \oplus \sigma(A) \oplus \sigma^2(A)$, where $A$ is a rank $1$ representation corresponding to a homomorphism $\rho_A \colon \Gamma \to \mathbb{F}_{q^3}^\times$ such that $\sigma(A) \ncong A$. Thus $|\mathcal{X}_3^{(2)}| = ( m_{q^3-1} - m_{q-1} )/3$. To compute $|\widetilde{\mathcal{X}}_3^{(2)}|$, note that the condition $A \sigma(A) \sigma(A)^2 \cong 1$ is equivalent to $\rho_A$ taking values in $\mu_{1+q+q^2} = \{ x \in \mathbb{F}_{q^3}^\times \; | \; x^{1+q+q^2} = 1 \}$. Since we assume $q = 1 \; ({\rm mod} \; 6)$, we find $|\mu_{1+q+q^2} \cap \mathbb{F}_q | = 3$. So $|\widetilde{\mathcal{X}}_3^{(2)}| = ( m_{q^2+q+1} - m_3 )/3$. The stabiliser in $GL_3(\mathbb{F}_q)$ consists of diagonal matrices $diag(x,\sigma(x) , \sigma^2(x))$, where $x \in \mathbb{F}_{q^3}^\times$, hence $s_2 = q^2+q+1$.\\

{\bf Strata $\mathcal{X}_3^{(3)}, \widetilde{\mathcal{X}}_3^{(3)}$.} The reductive representations of the form $A_2 \oplus A_1$, where $A_2$ is an absolutely irreducible representation of rank $2$ and $A_1$ has rank $1$. Then $\mathcal{X}_3^{(3)} = m_{q-1}|\mathcal{X}_2^{(1)}|$, $\widetilde{\mathcal{X}}_3^{(3)} = |\mathcal{X}_2^{(1)}|$ and $s_3 = q-1$.\\

{\bf Strata $\mathcal{X}_3^{(4)}, \widetilde{\mathcal{X}}_3^{(4)}$.} The reductive representations of the form $A \oplus \sigma(A) \oplus B$, where $B$ is a rank $1$ representation over $\mathbb{F}_q$ and $A$ is a rank $1$ representation over $\mathbb{F}_{q^2}$ with $\sigma(A) \ncong A$. Thus $|\mathcal{X}_3^{(4)}| = ( m_{q^2-1} - m_{q-1} )m_{q-1}/2$, $|\widetilde{\mathcal{X}}_3^{(4)}| = (m_{q^2-1} - m_{q-1})/2$. The stabiliser in $GL_3(\mathbb{F}_q)$ consists of diagonal matrices of the form $diag(x,\sigma(x),y)$, where $x \in \mathbb{F}_{q^2}^\times$, $y \in \mathbb{F}_q^\times $, so $s_4 = q^2-1$.\\

{\bf Strata $\mathcal{X}_3^{(5)}, \widetilde{\mathcal{X}}_3^{(5)}$.} The reductive representations of the form $A \oplus B \oplus C$, where $A,B,C$ are distinct rank $1$ representations over $\mathbb{F}_q$. Thus $|\mathcal{X}_3^{(5)}| = \binom{m_{q-1}}{3}$, $|\widetilde{\mathcal{X}}_3^{(5)}| = (m_{q-1}^2 -3m_{q-1} +2m_3)/6$ (by Lemma \ref{lemcountsl3}), and $s_5 = (q-1)^2$.\\

{\bf Strata $\mathcal{X}_3^{(6)}, \widetilde{\mathcal{X}}_3^{(6)}$.} The reductive representations of the form $A \oplus A \oplus B$, where $A,B$ are distinct rank $1$ representations over $\mathbb{F}_q$. Thus $|\mathcal{X}_3^{(6)}| = m_{q-1}(m_{q-1}-1)$, $\widetilde{\mathcal{X}}_3^{(6)} = m_{q-1}-m_3$, $s_6 = |GL_2(\mathbb{F}_q)| = q(q+1)(q-1)^2$.\\

{\bf Strata $\mathcal{X}_3^{(7)}, \widetilde{\mathcal{X}}_3^{(7)}$.} The reductive representations of the form $A \oplus A \oplus A$, where $A$ is a rank $1$ representation over $\mathbb{F}_q$. Thus $|\mathcal{X}_3^{(7)}| = m_{q-1}$, $|\widetilde{\mathcal{X}}_3^{(7)}| = m_3$, $s_7 = |PGL_3(\mathbb{F}_q)| = q^3(q^2-1)(q^3-1)$.\\

The strata for $8 \le i \le 11$ are the reducible representations which contain exactly one proper, non-trivial invariant subspace.\\

{\bf Strata $\mathcal{X}_3^{(8)}, \widetilde{\mathcal{X}}_3^{(8)}$.} The non-trivial extensions $A_2 \to E \to A_1$, where $A_2$ is rank $2$ absolutely irreducible and $A_1$ has rank $1$. Note that $dim( H^1( \Gamma , Hom(A_1 , A_2) ) ) = -2\chi(\Sigma_g) = 4g-4$. Thus $|\mathcal{X}_3^{(8)}| = m_{q-1}|\mathcal{X}_2^{(1)}| [4g-4]_q$, $|\widetilde{\mathcal{X}}_3^{(8)}| = |\mathcal{X}_2^{(1)}|[4g-4]_q$, $s_8 = 1$.\\

{\bf Strata $\mathcal{X}_3^{(9)}, \widetilde{\mathcal{X}}_3^{(9)}$.} The non-trivial extensions $A_1 \to E \to A_2$, where $A_2$ is rank $2$ absolutely irreducible and $A_1$ has rank $1$. These are the dual representations to $\mathcal{X}_3^{(8)}, \widetilde{\mathcal{X}}_3^{(8)}$. Thus $|\mathcal{X}_3^{(9)}| = |\mathcal{X}_3^{(8)}|$, $|\widetilde{\mathcal{X}}_3^{(9)}| = |\widetilde{\mathcal{X}}_3^{(8)}|$ and $s_9 = s_8$.\\

{\bf Strata $\mathcal{X}_3^{(10)}, \widetilde{\mathcal{X}}_3^{(10)}$.} The non-trivial extensions $A_2 \to E \to A_1$, where $A_1$ has rank $1$ and $A_2 = A \oplus \sigma(A)$, where $A$ is a rank $1$ representation over $\mathbb{F}_{q^2}$ with $\sigma(A) \ncong A$. For a fixed $A_1,A$, these are classified by the projectivisation over $\mathbb{F}_{q^2}$ of the $\mathbb{F}_{q^2}$-vector space $H^1( \Gamma , Hom(A_1 , A))$ which has dimension $2g-2$. Thus $|\mathcal{X}_3^{(10)}| = \frac{1}{2} m_{q-1}( m_{q^2-1} - m_{q-1}) [2g-2]_{q^2}$, $|\widetilde{\mathcal{X}}_3^{(10)}| = \frac{1}{2}( m_{q^2-1} - m_{q-1}) [2g-2]_{q^2}$, $s_{11} = 1$.\\

{\bf Strata $\mathcal{X}_3^{(11)}, \widetilde{\mathcal{X}}_3^{(11)}$.}  The non-trivial extensions $A_1 \to E \to A_2$, where $A_1$ has rank $1$ and $A_2 = A \oplus \sigma(A)$, where $A$ is a rank $1$ representation over $\mathbb{F}_{q^2}$ with $\sigma(A) \ncong A$. These are the dual representations to $\mathcal{X}_3^{(10)}, \widetilde{\mathcal{X}}_3^{(10)}$.\\

The strata for $12 \le i \le 16$ are the decomposable, non-reductive representations.\\

{\bf Strata $\mathcal{X}_3^{(12)}, \widetilde{\mathcal{X}}_3^{(12)}$.} Representations $A \oplus E$, where $E$ is a non-trivial extension $B \to E \to C$ and $A,B,C$ are distinct rank $1$ representations. Thus $|\mathcal{X}_3^{(12)}| = m_{q-1}(m_{q-1}-1)(m_{q-1}-2)[2g-2]_q$, $|\widetilde{\mathcal{X}}_3^{(12)}| = (m_{q-1}^2 -3m_{q-1} +2m_3)[2g-2]_q$ (using Lemma \ref{lemcountsl3}), and $s_{12} = (q-1)$.\\

{\bf Strata $\mathcal{X}_3^{(13)}, \widetilde{\mathcal{X}}_3^{(13)}$.} Representations $A \oplus E$, where $E$ is a non-trivial extension $A \to E \to C$ and $A,C$ are distinct rank $1$ representations. Thus $|\mathcal{X}_3^{(13)}| = m_{q-1}(m_{q-1}-1)[2g-2]_q$, $|\widetilde{\mathcal{X}}_3^{(13)}| = (m_{q-1}-m_3)[2g-2]_q$, $s_{13} = q(q-1)$.\\

{\bf Strata $\mathcal{X}_3^{(14)}, \widetilde{\mathcal{X}}_3^{(14)}$.} Representations $A \oplus E$, where $E$ is a non-trivial extension $B \to E \to A$ and $A,B$ are distinct rank $1$ representations. These representations are the duals to $\mathcal{X}_3^{(14)}, \widetilde{\mathcal{X}}_3^{(14)}$.\\

{\bf Strata $\mathcal{X}_3^{(15)}, \widetilde{\mathcal{X}}_3^{(15)}$.} Representations $A \oplus E$, where $E$ is a non-trivial extension $B \to E \to B$ and $A,B$ are distinct rank $1$ representations. Thus $|\mathcal{X}_3^{(15)}| = m_{q-1}(m_{q-1}-1)[2g]_q$, $|\widetilde{\mathcal{X}}_3^{(15)}| = (m_{q-1}-m_3)[2g]_q$, $s_{15} = q(q-1)$.\\

{\bf Strata $\mathcal{X}_3^{(16)}, \widetilde{\mathcal{X}}_3^{(16)}$.} Representations $A \oplus E$, where $E$ is a non-trivial extension $A \to E \to A$ and $A$ is a rank $1$ representation. Thus $|\mathcal{X}_3^{(16)}| = m_{q-1}[2g]_q$, $|\widetilde{\mathcal{X}}_3^{(16)}| = m_3[2g]_q$. We now compute the stabiliser. Identify $A \oplus E$ with the vector space $\mathbb{F}_q^3$ equipped with an action of $\Gamma$. Let $(1,0,0)$ span $A \oplus 0$ and $(0,1,0) , (0,0,1)$ span $0 \oplus E$. Further suppose that $(0,1,0)$ spans the invariant subspace $A \subset E$. Using $Hom(A \oplus E , A \oplus E) = Hom(A,A) \oplus Hom(A,E) \oplus Hom(E,A) \oplus Hom(E,E)$, we see that any invariant endomorphism $\phi : A \oplus E \to A \oplus E$ has the form
\begin{equation*}
\phi = \left[\begin{matrix} a & 0 & b \\ c & d & e \\ 0 & 0 & d \end{matrix} \right]
\end{equation*}
for some $a,b,c,d,e \in \mathbb{F}_q$. The determinant of this matrix is $ad^2$, hence it is an isomorphism if and only if $a$ and $d$ are non-zero. Thus the stabiliser in $GL_3(\mathbb{F}_q)$ has order $q^3(q-1)^2$ and $s_{16} = q^3(q-1)$.\\

The strata for $17 \le i \le 20$ are the representations with at least two rank $1$ invariant subspaces and exactly one rank $2$ invariant subspace. It follows that the rank $2$ invariant subspace is decomposable.\\

{\bf Strata $\mathcal{X}_3^{(17)}, \widetilde{\mathcal{X}}_3^{(17)}$.} The non-trivial extensions $A \oplus B \to E \to C$, where $A,B,C$ are distinct rank $1$ representations, such that $A \to E/B \to C$ and $B \to E/A \to C$ do not split. Thus $|\mathcal{X}_3^{(17)}| = \frac{1}{2}m_{q-1}(m_{q-1}-1)(m_{q-1}-2)[2g-2]_q^2$, $|\widetilde{\mathcal{X}}_3^{(17)}| = \frac{1}{2}(m_{q-1}^2 - 3m_{q-1}+2m_3)[2g-2]_q^2$, $s_{17} = 1$.\\

{\bf Strata $\mathcal{X}_3^{(18)}, \widetilde{\mathcal{X}}_3^{(18)}$.} The non-trivial extensions $A \oplus A \to E \to C$, where $A,C$ are distinct rank $1$ representations, such that for every invariant rank $1$ subspace $\iota \colon A \to A \oplus A$, the induced sequence $(A \oplus A)/\iota(A) \to E/\iota(A) \to C$ does not split. For given representations $A,C$, such extensions correspond to classes $\xi \in H^1(\Gamma , Hom(C , A \oplus A))$ with the following property. Let $V = H^1(\Gamma , Hom(C,A))$, which is a $2g-2$-dimensional vector space over $\mathbb{F}_q$. Then $\xi \in V \oplus V \cong Hom( \mathbb{F}_q^2 , V)$. The condition on $\xi$ is that it is injective when viewed as a map $\xi \colon \mathbb{F}_q^2 \to V$. Two such $\xi$ define isomorphic extensions if and only if they lie in the same orbit of $GL_2(\mathbb{F}_q)$ acting on $Hom(\mathbb{F}_q^2 , V)$ through the standard action on $\mathbb{F}_q^2$. Thus the number of such extensions is $| Gr_q( 2 , 2g-2)|$, where $Gr_q(m,n)$ is the Grassmannian of $m$-dimensional subspaces of $\mathbb{F}_q^n$. This is given by $|Gr_q(2,2g-2)| = \frac{(q^{2g-2}-1)(q^{2g-3}-1)}{(q-1)(q^2-1)}$. Thus $|\mathcal{X}_3^{(18)}| = m_{q-1}(m_{q-1}-1)\frac{1}{(q+1)}[2g-2]_q[2g-3]_q$, $|\widetilde{\mathcal{X}}_3^{(18)}| = (m_{q-1}-m_3)\frac{1}{(q+1)}[2g-2]_q[2g-3]_q$, $s_{18} = 1$.\\

{\bf Strata $\mathcal{X}_3^{(19)}, \widetilde{\mathcal{X}}_3^{(19)}$.} The non-trivial extensions $A \oplus B \to E \to A$, where $A,B$ are distinct rank $1$ representations, such that $A \to E/B \to A$ and $B \to E/A \to A$ do not split. Thus $|\mathcal{X}_3^{(19)}| = m_{q-1}(m_{q-1}-1)[2g-2]_q[2g]_q$, $|\widetilde{\mathcal{X}}_3^{(19)}| = (m_{q-1}-m_3)[2g-2]_q[2g]_q$, $s_{19} = q$.\\

{\bf Strata $\mathcal{X}_3^{(20)}, \widetilde{\mathcal{X}}_3^{(20)}$.} The non-trivial extensions $A \oplus A \to E \to A$, where $A$ is a rank $1$ representation, such that for every invariant rank $1$ subspace $\iota \colon A \to A \oplus A$, the induced sequence $(A \oplus A)/\iota(A) \to E/\iota(A) \to A$ does not split. By the same reasoning as used for $\mathcal{X}_3^{(18)}$, we find that $|\mathcal{X}_3^{(20)}| = m_{q-1}\frac{1}{(q+1)}[2g]_q[2g-1]_q$, $|\widetilde{\mathcal{X}}_3^{(20)}| = m_3 \frac{1}{(q+1)}[2g]_q[2g-1]_q$, $s_{20} = q^2$.\\

The strata for $21 \le i \le 24$ are the representations $E$ with at least two rank $2$ invariant subspaces and exactly one rank $1$ invariant subspace $C$. It follows that $E/C$ is decomposable. These are the duals of the strata $17 \le i \le 20$.\\

{\bf Strata $\mathcal{X}_3^{(21)}, \widetilde{\mathcal{X}}_3^{(21)}$.} The extensions $C \to E \to A \oplus B$, where $A,B,C$ are distinct rank $1$ representations and for which the extension class in $H^1(\Gamma , Hom(A \oplus B , C))$ restricts to non-trivial classes in $H^1(\Gamma , Hom(A,C))$ and $H^1(\Gamma , Hom(B,C))$. These are the dual representations of $\mathcal{X}_3^{(17)}, \widetilde{\mathcal{X}}_3^{(17)}$.\\

{\bf Strata $\mathcal{X}_3^{(22)}, \widetilde{\mathcal{X}}_3^{(22)}$.} The extensions $C \to E \to A \oplus A$, where $A,C$ are distinct rank $1$ representations and for which the extension class in $H^1(\Gamma , Hom(A \oplus A , C))$ restricts to a non-trivial class in $H^1(\Gamma , Hom(A,C))$ for every invariant rank $1$ subspace $\iota : A \to A \oplus A$. These are the dual representations of $\mathcal{X}_3^{(18)}, \widetilde{\mathcal{X}}_3^{(18)}$.\\

{\bf Strata $\mathcal{X}_3^{(23)}, \widetilde{\mathcal{X}}_3^{(23)}$.} The extensions $A \to E \to A \oplus B$, where $A,B$ are distinct rank $1$ representations and for which the extension class in $H^1(\Gamma , Hom(A \oplus B , A))$ restricts to non-trivial classes in $H^1(\Gamma , Hom(A,A))$ and $H^1(\Gamma , Hom(B,A))$. These are the dual representations of $\mathcal{X}_3^{(19)}, \widetilde{\mathcal{X}}_3^{(19)}$.\\

{\bf Strata $\mathcal{X}_3^{(24)}, \widetilde{\mathcal{X}}_3^{(24)}$.} The extensions $A \to E \to A \oplus A$, where $A$ is a rank $1$ representation and for which the extension class in $H^1(\Gamma , Hom(A \oplus A , A))$ restricts to a non-trivial class in $H^1(\Gamma , Hom(A,A))$ for every invariant rank $1$ subspace $\iota \colon A \to A \oplus A$. These are the dual representations of $\mathcal{X}_3^{(20)}, \widetilde{\mathcal{X}}_3^{(20)}$.\\

The strata for $25 \le i \le 30$ are the indecomposable representations containing exactly one rank $1$ invariant subspace and exactly one rank $2$ invariant subspace. They are given by iterated non-trivial extensions.\\

{\bf Strata $\mathcal{X}_3^{(25)}, \widetilde{\mathcal{X}}_3^{(25)}$.} The non-trivial extensions $A \to E \to F$, where $F$ is a non-trivial extension $B \to F \to C$ and $A,B,C$ are distinct rank $1$ representations. Let $A,B,C$ be fixed. The number of isomorphism classes of non-trivial extensions $B \to F \to C$ is $[2g-2]_q$. Fix such an extension. An extension $A \to E \to F$ is given by an element $\xi \in H^1(\Gamma , Hom(F , A))$. The long exact sequence in cohomology associated to $Hom(C,A) \to Hom(F,A) \to Hom(B,A)$, together with the vanishing of $H^0(\Gamma , Hom(B,A))$ and $H^2(\Gamma , Hom(C,A)) \cong H^0(\Gamma , Hom(A,C))^*$, gives a short exact sequence:
\begin{equation*}
0 \to H^1(\Gamma , Hom(C,A)) \to H^1(\Gamma , Hom(F,A)) \to H^1(\Gamma , Hom(B,A)) \to 0.
\end{equation*}
Thus, up to isomorphism, the number of extensions $A \to E \to F$ which restrict to non-trivial classes in $H^1(\Gamma , Hom(B,A))$ is $[2g-2]_q q^{2g-2}$. Then $|\mathcal{X}_3^{(25)}| = m_{q-1}(m_{q-1}-1)(m_{q-1}-2)[2g-2]_q^2 q^{2g-2}$, $|\widetilde{\mathcal{X}}_3^{(25)}| = (m_{q-1}^2 - 3m_{q-1}+2m_3)[2g-2]_q^2 q^{2g-2}$, $s_{25} = 1$.\\

{\bf Strata $\mathcal{X}_3^{(26)}, \widetilde{\mathcal{X}}_3^{(26)}$.} The non-trivial extensions $A \to E \to F$, where $F$ is a non-trivial extension $A \to F \to C$ and $A,C$ are distinct rank $1$ representations. Let $A,C$ be fixed. The number of non-trivial extensions $A \to F \to C$ is $[2g-2]_q$. Fix such an extension. An extension $A \to E \to F$ is given by an element $\xi \in H^1(\Gamma , Hom(F,A))$. The long exact sequence in cohomology associated to $Hom(C,A) \to Hom(F,A) \to Hom(A,A)$ together with $H^0(\Gamma , Hom(A,A)) \cong \mathbb{F}_q$, $H^0(\Gamma , Hom(F,A)) = 0$, $H^2( \Gamma , Hom(C,A)) = 0$ gives a long exact sequence:
\begin{equation*}
0 \to \mathbb{F}_q \to H^1(\Gamma , Hom(C,A)) \to H^1(\Gamma , Hom(F,A)) \to H^1(\Gamma , Hom(A,A)) \to 0.
\end{equation*}
Thus, up to isomorphism, the number of extensions $A \to E \to F$ which restrict to non-trivial classes in $H^1(\Gamma , Hom(A,A))$ is $[2g]_q q^{2g-3}$. Then $|\mathcal{X}_3^{(26)}| = m_{q-1}(m_{q-1}-1)[2g-2]_q [2g]_q q^{2g-3}$, $|\widetilde{\mathcal{X}}_3^{(26)}| = (m_{q-1} - m_3)[2g-2]_q [2g]_q q^{2g-3}$, $s_{26} = 1$.\\

{\bf Strata $\mathcal{X}_3^{(27)}, \widetilde{\mathcal{X}}_3^{(27)}$.} The non-trivial extensions $A \to E \to F$, where $F$ is a non-trivial extension $B \to F \to B$ and $A,B$ are distinct rank $1$ representations. These are the dual representations of $\mathcal{X}_3^{(26)}, \widetilde{\mathcal{X}}_3^{(26)}$.\\ 

{\bf Strata $\mathcal{X}_3^{(28)}, \widetilde{\mathcal{X}}_3^{(28)}$.} The non-trivial extensions $A \to E \to F$, where $F$ is a non-trivial extension $B \to F \to A$ and $A,B$ are distinct rank $1$ representations. Let $A,B$ be fixed. The number of non-trivial extensions $B \to F \to A$ is $[2g-2]_q$. Fix such an extension. An extension $A \to E \to F$ is given by an element $\xi \in H^1(\Gamma , Hom(F,A))$. The long exact sequence in cohomology associated to $Hom(A,A) \to Hom(F,A) \to Hom(B,A)$ together with $H^0(\Gamma , Hom(B,A)) = 0$, $H^2(\Gamma , Hom(F,A)) \cong H^0(\Gamma , Hom(A,F))^* = 0$, $H^2( \Gamma , Hom(A,A)) \cong H^0(\Gamma , Hom(A,A))^* \cong \mathbb{F}_q$ gives a long exact sequence:

\begin{equation*}
0 \to H^1(\Gamma , Hom(A,A)) \to H^1(\Gamma , Hom(F,A)) \to H^1(\Gamma , Hom(B,A)) \to \mathbb{F}_q \to 0.
\end{equation*}

Thus, up to isomorphism, the number of extensions $A \to E \to F$ which restrict to non-trivial classes in $H^1(\Gamma , Hom(B,A))$ is $[2g-3]_q q^{2g}$. Then $|\mathcal{X}_3^{(28)}| = m_{q-1}(m_{q-1}-1)[2g-2]_q [2g-3]_q q^{2g}$, $|\widetilde{\mathcal{X}}_3^{(28)}| = (m_{q-1} - m_3)[2g-2]_q [2g-3]_q q^{2g}$, $s_{28} = q$.\\

{\bf Strata $\mathcal{X}_3^{(29)}, \widetilde{\mathcal{X}}_3^{(29)}$.} The non-trivial extensions $A \to E \to F$, where $F$ is a non-trivial extension $A \to F \to A$ and $A$ is a rank $1$ representation satisfying the following condition. Let $\beta \in H^1(\Gamma , Hom(A,A))$ be the extension class of $A \to F \to A$, let $\xi \in H^1(\Gamma , Hom(F,A))$ be the extension class of $A \to E \to F$ and let $\alpha \in H^1(\Gamma , Hom(A,A))$ be the restriction of $\xi$ under the inclusion $A \to F$. Then we require that $\alpha, \beta \in H^1(\Gamma , Hom(A,A))$ are not multiples of one another. Fix the representation $A$. The number of non-trivial extensions $A \to F \to A$ is given by $[2g]_q$. Fix an extension class $\beta \in H^1(\Gamma , Hom(A,A))$ representing such an extension $F$. The long exact sequence in cohomology associated to $Hom(A,A) \to Hom(F,A) \to Hom(A,A)$ takes the form:
\begin{equation}\label{equles}
0 \to \mathbb{F}_q \buildrel \beta \over \longrightarrow H^1(\Gamma , Hom(A,A) \to H^1(\Gamma , Hom(F,A)) \to H^1(\Gamma , Hom(A,A)) \buildrel \cup \beta \over \longrightarrow \mathbb{F}_q \to 0
\end{equation}
where $\cup \beta \colon H^1(\Gamma , Hom(A,A)) \to \mathbb{F}_q$ is the map sending a class $\alpha \in H^1(\Gamma , Hom(F,A))$ to the cup product $\alpha \cup \beta \in H^2(\Gamma , Hom(A,A)) \cong \mathbb{F}_q$. Let $V = H^1(\Gamma , Hom(A,A))$, which is a $2g$-dimensional vector space over $\mathbb{F}_q$ and $\beta$ is an element of $V$. The cup product defines an alternating bilinear form $V \otimes V \to \mathbb{F}_q$. The long exact sequence (\ref{equles}) gives a non-canonical identification of $H^1(\Gamma , Hom(F,A))$ with the set of pairs $\xi = ( \gamma , \alpha ) \in V \oplus V$ satisfying $\alpha \cup \beta = 0$, modulo the subspace spanned by $(\beta, 0)$. Under this identification the restriction map $H^1(\Gamma , Hom(F,A)) \to H^1(\Gamma , Hom(A,A))$ sends $(\gamma , \alpha )$ to $\alpha$. We seek to count the number of isomorphism classes of non-trivial extensions $A \to E \to F$ defined by extension classes $\xi = (\gamma , \alpha ) \in H^1(\Gamma , Hom(F,A))$. For this we note that two pairs $\xi_1, \xi_2$ define isomorphic extensions if and only if they lie in the same orbit of the natural action of $Aut(F)$ on $H^1(\Gamma , Hom(F,A))$, where $Aut(F)$ is the group of automorphisms of $F$. At the level of pairs $(\gamma , \alpha)$, the action of $Aut(F)$ is generated by rescaling $(\gamma , \alpha) \mapsto (c \gamma , c\alpha)$, $c \in \mathbb{F}_q^\times$ and shifts $(\gamma , \alpha) \mapsto (\gamma + t\alpha , \alpha)$, for $t \in \mathbb{F}_q$. It follows that the number of isomorphism classes of pairs $\xi = (\gamma , \alpha)$ with $\alpha$ not proportional to $\beta$ is $[2g-2]_q q^{2g-1}$. Thus $|\mathcal{X}_3^{(29)}| = m_{q-1}[2g]_q [2g-2]_q q^{2g-1}$, $|\widetilde{\mathcal{X}}_3^{(29)}| = m_3 [2g]_q [2g-2]_q q^{2g-1}$, $s_{29} = q$.\\

{\bf Strata $\mathcal{X}_3^{(30)}, \widetilde{\mathcal{X}}_3^{(30)}$.} This is the same as for $\mathcal{X}_3^{(29)}, \widetilde{\mathcal{X}}_3^{(29)}$, except we assume that $\alpha,\beta$ are proportional. It follows that for given $A$ and $\beta$, the number of isomorphism classes of pairs $\xi = (\gamma , \alpha)$ with $\alpha$ proportional to $\beta$ is $q^{2g-1}$. Thus $|\mathcal{X}_3^{(30)}| = m_{q-1}[2g]_q q^{2g-1}$, $|\widetilde{\mathcal{X}}_3^{(30)}| = m_3 [2g]_q q^{2g-1}$, $s_{30} = q^2$.\\

The calculations in this section are summarised in Table \ref{figrank3}.

\begin{table}
\begin{equation*}
\renewcommand{\arraystretch}{1.5}
\centerline{
\begin{tabular}{|c|c|c|c|}
\hline
$i$ & $|\mathcal{X}_3^{(i)}|$ & $|\widetilde{\mathcal{X}}_3^{(i)}|$ & $s_i$ \\
\hline
$2$ & $\frac{1}{3}( m_{q^3-1} - m_{q-1} )$ & $\frac{1}{3}( m_{q^2+q+1} - m_3 )$ & $q^2+q+1$ \\
$3$ & $m_{q-1}|\mathcal{X}_2^{(1)}|$ & $|\mathcal{X}_2^{(1)}|$ & $(q-1)$ \\
$4$ & $\frac{1}{2}( m_{q^2-1} - m_{q-1} )m_{q-1}$ & $\frac{1}{2}(m_{q^2-1} - m_{q-1})$ & $(q^2-1)$ \\
$5$ & $\binom{m_{q-1}}{3}$ & $\frac{1}{6}(m_{q-1}^2 -3m_{q-1} +2m_3)$ & $(q-1)^2$ \\
$6$ & $m_{q-1}(m_{q-1}-1)$ & $m_{q-1}-m_3$ & $q(q+1)(q-1)^2$ \\
$7$ & $m_{q-1}$ & $m_3$ & $q^3(q^2-1)(q^3-1)$ \\
$8,9$ & $m_{q-1}|\mathcal{X}_2^{(1)}| [4g-4]_q$ & $|\mathcal{X}_2^{(1)}|[4g-4]_q$ & $1$ \\
$10,11$ & $\frac{1}{2} m_{q-1}( m_{q^2-1} - m_{q-1}) [2g-2]_{q^2}$ & $\frac{1}{2}( m_{q^2-1} - m_{q-1}) [2g-2]_{q^2}$ & $1$ \\
$12$ & $m_{q-1}(m_{q-1}-1)(m_{q-1}-2)[2g-2]_q$ & $(m_{q-1}^2 -3m_{q-1} +2m_3)[2g-2]_q$ & $(q-1)$ \\
$13,14$ & $m_{q-1}(m_{q-1}-1)[2g-2]_q$ & $(m_{q-1}-m_3)[2g-2]_q$ & $q(q-1)$ \\
$15$ & $m_{q-1}(m_{q-1}-1)[2g]_q$ & $(m_{q-1}-m_3)[2g]_q$ & $q(q-1)$ \\
$16$ & $m_{q-1}[2g]_q$ & $m_3[2g]_q$ & $q^3(q-1)$ \\
$17,21$ & $\frac{1}{2}m_{q-1}(m_{q-1}-1)(m_{q-1}-2)[2g-2]_q^2$ & $\frac{1}{2}(m_{q-1}^2 - 3m_{q-1}+2m_3)[2g-2]_q^2$ & $1$ \\
$18,22$ & $m_{q-1}(m_{q-1}-1)\frac{1}{(q+1)}[2g-2]_q[2g-3]_q$ & $(m_{q-1}-m_3)\frac{1}{(q+1)}[2g-2]_q[2g-3]_q$ & $1$ \\
$19,23$ & $m_{q-1}(m_{q-1}-1)[2g-2]_q[2g]_q$ & $(m_{q-1}-m_3)[2g-2]_q[2g]_q$ & $q$ \\
$20,24$ & $m_{q-1}\frac{1}{(q+1)}[2g]_q[2g-1]_q$ & $m_3 \frac{1}{(q+1)}[2g]_q[2g-1]_q$ & $q^2$ \\
$25$ & $m_{q-1}(m_{q-1}-1)(m_{q-1}-2)[2g-2]_q^2 q^{2g-2}$ & $(m_{q-1}^2 - 3m_{q-1}+2m_3)[2g-2]_q^2 q^{2g-2}$ & $1$ \\
$26,27$ & $m_{q-1}(m_{q-1}-1)[2g-2]_q [2g]_q q^{2g-3}$ & $(m_{q-1} - m_3)[2g-2]_q [2g]_q q^{2g-3}$ & $1$ \\
$28$ & $m_{q-1}(m_{q-1}-1)[2g-2]_q [2g-3]_q q^{2g}$ & $(m_{q-1} - m_3)[2g-2]_q [2g-3]_q q^{2g}$ & $q$ \\
$29$ & $m_{q-1}[2g]_q [2g-2]_q q^{2g-1}$ & $m_3 [2g]_q [2g-2]_q q^{2g-1}$ & $q$ \\
$30$ & $m_{q-1}[2g]_q q^{2g-1}$ & $m_3 [2g]_q q^{2g-1}$ & $q^2$ \\
\hline
\end{tabular}}
\end{equation*}
\caption{Sizes of strata and their stabilisers in the rank $3$ case.}\label{figrank3}
\end{table}

\section{Counting $|Hom(\Gamma , G)|$}\label{seccountinghoms}

Let $G$ be a finite group and let $Hom(\Gamma,G)$ be the set of homomorphisms from $\Gamma$ to $G$. In this section we recall that for certain $\Gamma$, one can compute the size of $Hom(\Gamma , G)$ using character theory. Let us introduce the following notation; for a function $f\colon G \to \mathbb{C}$, we define
\begin{equation*}
\int_G f(x)dx = \frac{1}{|G|} \sum_{x \in G} f(x).
\end{equation*}

The main tools we use are the following results:

\begin{proposition}[\cite{eghlsvy}, Theorem 4.5.4]
Let $\delta \colon G \to \mathbb{C}$ be given by 
\begin{equation*}
\delta(g) = \begin{cases} 1 & \text{if } $g=1$, \\ 0 & \text{otherwise}. \end{cases}
\end{equation*}
Then
\begin{equation*}
\delta = \frac{1}{|G|} \sum_{\chi} \chi(1) \chi,
\end{equation*}
where the sum is over the irreducible characters of $G$.
\end{proposition}

\begin{proposition}[\cite{harv}, \textsection 2.3]
Let $\chi$ be any irreducible character of $G$ and $z$ any element of $G$. Let $w(x_1,\dots , x_n)$ denote a word in $x_1,...,x_n$. Then
\begin{equation*}
\int_{G^n} \chi( w(x_1,\dots , x_n ) z )dx_1 \dots dx_n = \frac{\chi(z)}{\chi(1)} \int_{G^n} \chi( w( x_1 , \dots , x_n) )dx_1 \dots dx_n.
\end{equation*}
\end{proposition}

\begin{proposition}[\cite{harv}, \textsection 2.3]
For any irreducible character $\chi$, we have
\begin{equation*}
\int_{G^2} \chi(xyx^{-1}y^{-1})dxdy = \frac{1}{\chi(1)}.
\end{equation*}
\end{proposition}

Let $\chi$ be an irreducible character of $G$ and $n$ a positive integer. Then we define the {\em $n$-th Frobenius-Schur indicator} $\nu_n(\chi)$ by
\begin{equation*}
\nu_n(\chi) = \int_G \chi( x^n) dx.
\end{equation*}
It is known that $\nu_n(\chi)$ is an integer \cite{isa}. When $n=2$, $\nu_2(\chi)$ is the usual Frobenius-Schur indicator and we have:
\begin{equation*}
\nu_2(\chi) = \begin{cases} $1$ & \text{if } \chi \text{ is real,} \\ $0$ & \text{if } \chi \text{ is complex,} \\ $-1$ & \text{if } \chi \text{ is quaternionic.} \end{cases}
\end{equation*}

\subsection{Fundamental groups of compact oriented surfaces}
Let $\Gamma = \pi_1(\Sigma)$, where $\Sigma$ is a compact oriented surface of genus $g$. Then $\Gamma$ has a presentation
\begin{equation*}
\Gamma = \langle a_1 , b_1 , \dots , a_g , b_g \; | \; [a_1,b_1] \dots [a_g,b_g] = 1 \rangle.
\end{equation*}
Hence
\begin{equation*}
\begin{aligned}
| Hom(\Gamma , G)| &= |G|^{2g} \int_{G^{2g}} \delta( [a_1,b_1]\dots [a_g,b_g]) da_1 \dots db_g \\
&= |G|^{2g-1} \sum_\chi \chi(1) \int_{G^{2g}} \chi( [a_1,b_1]\dots [a_g,b_g]) da_1 \dots db_g \\
&= |G|^{2g-1} \sum_\chi \chi(1)^{2-g} \left( \int_{G^2} \chi( [a,b] ) dadb \right)^g \\
&= |G|^{2g-1} \sum_\chi \chi(1)^{2-2g},
\end{aligned}
\end{equation*}
and we have recovered the well known formula:
\begin{equation}\label{equhomcos}
\frac{|Hom(\Gamma , G)|}{|G|} = \sum_\chi \left( \frac{|G|}{\chi(1)} \right)^{2g-2}.
\end{equation}

\subsection{Fundamental groups of compact non-orientable surfaces}
Let $\Gamma = \pi_1(\Sigma)$, where $\Sigma$ is a compact non-orientable surface of Euler characteristic $e$. Then $\Sigma = \#^k \mathbb{RP}^2$, where $k=2-e$. So $\Gamma$ has a presentation
\begin{equation*}
\Gamma = \langle a_1 , a_2 \dots , a_k \; | \; a_1^2 a_2^2 \dots a_k^2 = 1 \rangle.
\end{equation*}
Hence
\begin{equation*}
\begin{aligned}
| Hom(\Gamma , G)| &= |G|^{k} \int_{G^{k}} \delta( a_1^2\dots a_k^2) da_1 \dots da_k \\
&= |G|^{k-1} \sum_\chi \chi(1) \int_{G^{k}} \chi( a_1^2\dots a_k^2) da_1 \dots da_k \\
&= |G|^{k-1} \sum_\chi \chi(1)^{2-k} \left( \int_{G^2} \chi( a^2 ) da \right)^k \\
&= |G|^{k-1} \sum_\chi \chi(1)^{2-k} \nu_2^k(\chi),
\end{aligned}
\end{equation*}
so we have found:
\begin{equation}\label{equhomcnos}
\frac{|Hom(\Gamma , G)|}{|G|} = \sum_\chi \left( \frac{|G|}{\chi(1)} \right)^{k-2}\nu_2^k(\chi).
\end{equation}

\subsection{Torus knots}
Let $a,b$ be coprime positive integers and set $\Gamma = \pi_1(S^3 \setminus K)$, where $K$ is an $(a,b)$-torus knot. Then $\Gamma$ is known to have the presentation
\begin{equation*}
\Gamma = \langle x,y \; | \; x^a = y^b \rangle.
\end{equation*}
Hence,
\begin{equation*}
\begin{aligned}
| Hom(\Gamma , G)| &= |G|^{2} \int_{G^{2}} \delta( x^a y^{-b}) dx dy \\
&= |G| \sum_\chi \chi(1) \int_{G^{2}} \chi( x^a y^{-b}) dxdy \\
&= |G| \sum_\chi \left( \int_{G} \chi( x^a ) dx \right) \left( \int_{G} \chi( y^{-b} ) dy \right)\\
&= |G| \sum_\chi \nu_a(\chi) \nu_b(\chi),
\end{aligned}
\end{equation*}
so we have found:
\begin{equation}\label{equhomtk}
\frac{|Hom(\Gamma , G)|}{|G|} = \sum_\chi \nu_a(\chi)\nu_b(\chi).
\end{equation}

\subsection{Characters of $GL_2(\mathbb{F}_q)$}

Assume $q$ is odd. For an abelian group $A$, we let $\widehat{A} = Hom(A , \mathbb{C}^*)$ denote the character group of $A$. The character table for $GL_2(\mathbb{F}_q)$ is \cite{dimi}, \cite{mer}:

\begin{equation*}
\renewcommand{\arraystretch}{1.4}
\begin{tabular}{|c|cccc|}
\hline
classes & $\left( \begin{matrix} a & 0 \\ 0 & a \end{matrix} \right)$ & $\left( \begin{matrix} a & 1 \\ 0 & a \end{matrix} \right)$ & $\left( \begin{matrix} a & 0 \\ 0 & b \end{matrix} \right)$ & $\left( \begin{matrix} x & 0 \\ 0 & x^q \end{matrix} \right)$ \\
&  &  & $a \neq b$ & $x \neq x^q$ \\
\hline
\# classes & $q-1$ & $q-1$ & $(q-1)(q-2)/2$ & $q(q-1)/2$ \\
class size & $1$ & $q^2-1$ & $q(q+1)$ & $q(q-1)$ \\
\hline
$R^G_T(\alpha,\beta)$ & $(q+1)\alpha(a)\beta(a)$ & $\alpha(a)\beta(a)$ &$\alpha(a)\beta(b) + \alpha(b)\beta(a)$ & $0$ \\
$-R^G_{T^s}(\omega)$ & $(q-1)\omega(a)$ & $-\omega(a)$ & $0$ & $-(\omega(x) + \omega(x^q))$ \\
$\sigma_{\alpha}(1)$ & $\alpha(a^2)$ & $\alpha(a^2)$ & $\alpha(ab)$ & $\alpha(x^{q+1})$ \\
$\sigma_{\alpha}(St_G)$ & $q\alpha(a^2)$ & $0$ & $\alpha(ab)$ & $-\alpha(x^{q+1})$ \\
\hline
\end{tabular}
\end{equation*}
\vspace{0.4cm}

In this table $a,b \in \mathbb{F}_q^\times$, $a \neq b$, $x \in \mathbb{F}_{q^2}^\times$, $x \neq x^q$, $\alpha,\beta \in \widehat{\mathbb{F}_q^\times}$, $\alpha \neq \beta$, $\omega \in \widehat{\mathbb{F}_{q^2}^\times}$, $\omega \neq \omega^q$. Swapping $a$ with $b$ or $x$ with $x^q$ gives the same conjugacy class. Similarly swapping $\alpha$ with $\beta$ or $\omega$ with $\omega^q$ gives the same character.

\subsection{Characters of $SL_2(\mathbb{F}_q)$}
Assume that $q$ is odd. Let $\alpha_0 \in \widehat{\mathbb{F}}_q^\times$ be the unique element of order $2$ and similarly let $\omega_0 \in \widehat{\mu}_{q+1}$ be the unique element of order $2$. The character table for $SL_2(\mathbb{F}_q)$ is \cite{dimi}, \cite{mer}:

\begin{equation*}
\renewcommand{\arraystretch}{1.4}
\begin{tabular}{|c|cccc|}
\hline
classes & $\left( \begin{matrix} a & 0 \\ 0 & a \end{matrix} \right)$ & $\left( \begin{matrix} a & b \\ 0 & a \end{matrix} \right)$ & $\left( \begin{matrix} a & 0 \\ 0 & a^{-1} \end{matrix} \right)$ & $\left( \begin{matrix} x & 0 \\ 0 & x^q \end{matrix} \right)$ \\
& $a = \pm 1$ & $a = \pm 1$ & $a \notin \{1,-1\}$ & $x \neq x^q$ \\
\hline
\# classes & $2$ & $4$ & $(q-3)/2$ & $(q-1)/2$ \\
class size & $1$ & $(q^2-1)/2$ & $q(q+1)$ & $q(q-1)$ \\
\hline
$R^G_T(\alpha)$ & $(q+1)\alpha(a)$ & $\alpha(a)$ &$\alpha(a) + \alpha(a^{-1})$ & $0$ \\
$\chi^{\pm}_{\alpha_0}$ & $\dfrac{(q+1)}{2}\alpha_0(a)$ & $\dfrac{\alpha_0(a)}{2}(1 \pm \varphi)$ & $\alpha_0(a)$ & $0$ \\
$-R^G_{T^s}(\omega)$ & $(q-1)\omega(a)$ & $-\omega(a)$ & $0$ & $-(\omega(x) + \omega(x^q))$ \\
$\chi^{\pm}_{\omega_0}$ & $\dfrac{(q-1)}{2}\omega_0(a)$ & $\dfrac{\omega_0(a)}{2}(-1 \pm \varphi)$ & $0$ & $-\omega_0(x)$ \\
$1$ & $1$ & $1$ & $1$ & $1$ \\
$St_G$ & $q$ & $0$ & $1$ & $-1$ \\
\hline
\end{tabular}
\end{equation*}
\vspace{0.4cm}

In this table $\mu_{q+1} = \{ x \in \mathbb{F}_{q^2}^\times \; | \; x^{q+1} = 1 \}$, $a \in \mathbb{F}_q^\times$, $x \in \mu_{q+1}$ with $x \neq x^q$, $\alpha \in \widehat{\mathbb{F}}_q^\times$ with $\alpha \neq 1 , \alpha_0$, $\omega \in \widehat{\mu}_{q+1}$ with $\omega \neq 1 , \omega_0$, $b$ is either $1$ or $y$, where $y \in \mathbb{F}_q^\times$ is some fixed non-square, $\varphi$ denotes the complex number $\alpha_0(ab)\sqrt{\alpha_0(-1)q}$. Note that replacing $a$ by $a^{-1}$ or $x$ by $x^{-1} = x^q$ gives a different representative for the same conjugacy class. Similarly, replacing $\alpha$ by $\alpha^{-1}$ or $\omega$ by $\omega^{-1}$ gives the same character.

\subsection{Frobenius-Schur indicators}

Assume $q = 1 \; ({\rm mod} \; 4)$ and $q = 1 \; ({\rm mod} \; a)$. By a direct computation from the character table, we find the Frobenius-Schur indicators for $GL_2(\mathbb{F}_q)$ are:

\begin{equation*}
\renewcommand{\arraystretch}{1.4}
\begin{tabular}{|c|cc|}
\hline
$\chi$ & $\nu_a(\chi)$ & $\nu_a(\chi)$ \\
& $a$ even & $a$ odd \\
\hline
$R^G_T(\alpha,\beta)$ & $\delta_{\alpha^a}\delta_{\beta^a} +\delta_{\alpha^{a/2}\beta^{a/2}} + \frac{(a-2)}{2}\delta_{\alpha^a\beta^a}$ & $\delta_{\alpha^a}\delta_{\beta^a} + \frac{(a-1)}{2}\delta_{\alpha^a \beta^a}$  \\
$-R^G_{T^s}(\omega)$ & $-\delta_{\omega^a} +\delta_{\omega^{(q+1)a/2}} + \frac{(a-2)}{2}\delta_{\omega^{(q+1)a}}$ & $-\delta_{\omega^a} + \frac{(a-1)}{2}\delta_{\omega^{(q+1)a}}$ \\
$\sigma_{\alpha}(1)$ & $\delta_{\alpha^a}$ & $\delta_{\alpha^a}$ \\
$\sigma_{\alpha}(St_G)$ & $\delta_{\alpha^a} +\frac{(a-2)}{2}\delta_{\alpha^{2a}}$ & $\frac{(a-1)}{2}\delta_{\alpha^{2a}}$ \\
\hline
\end{tabular}
\end{equation*}
\vspace{0.1cm}

In this table, $\delta_{\alpha} = 1$ if $\alpha = 1$, $\delta_{\alpha} = 0$ if $\alpha \neq 1$.\\

Similarly the Frobenius-Schur indicators for $SL_2(\mathbb{F}_q)$ are as follows. In this table we calculate $\nu_a(\chi)$ under the assumption that $q = 1 \; ({\rm mod} \; 4)$ and $q = 1 \; ({\rm mod} \; a)$.

\begin{equation*}
\renewcommand{\arraystretch}{1.4}
\begin{tabular}{|c|cc|}
\hline
$\chi$ & $\nu_a(\chi)$ & $\nu_a(\chi)$ \\
& $a$ even & $a$ odd \\
\hline
$R^G_T(\alpha)$ & $\delta_{\alpha^a} -1 + \dfrac{a}{2}(1+\alpha(-1))$ & $\delta_{\alpha^a} + \dfrac{(a-1)}{2}(1 + \alpha(-1))$  \\
$\chi^{\pm}_{\alpha_0}$ & $\dfrac{a}{2}$ & $\dfrac{(a-1)}{2}$  \\
$-R^G_{T^s}(\omega)$ & $-1 + \dfrac{a}{2}(1 + \omega(-1))$ & $\dfrac{(a-1)}{2}(1 + \omega(-1))$ \\
$\chi^{\pm}_{\omega_0}$ & $-1$ & $0$ \\
$1$ & $1$ & $1$ \\
$St_G$ & $a-1$ & $a-1$ \\
\hline
\end{tabular}
\end{equation*}
\vspace{0cm}
\subsection{Case of $GL_3(\mathbb{F}_q)$}\label{sechomsgl3}

We assume $q = 1 \; ({\rm mod} \; 3)$. In this section we will compute $|Hom(\Gamma , GL_3(\mathbb{F}_q))|$ in the case that $\Gamma = \pi_1(\Sigma_g)$ is the fundamental group of a compact oriented surface of genus $g$. From (\ref{equhomcos}), we have
\begin{equation}\label{equgl3hom}
\frac{|Hom(\Gamma , GL_3(\mathbb{F}_q))|}{|GL_3(\mathbb{F}_q)|} = \sum_\chi \left( \frac{|GL_3(\mathbb{F}_q)|}{\chi(1)} \right)^{2g-2},
\end{equation}
where the sum is over the irreducible characters of $GL_3(\mathbb{F}_q)$. To compute this we need to briefly review some facts about the irreducible characters of the general linear groups over finite fields \cite{mac}, \cite{harv} \cite{gre}.\\

Let $\mathcal{P}_m$ be the set of all partitions $\lambda$ of $m$ and write $m = |\lambda|$. Let $\mathcal{P} = \cup_{m \ge 0} \mathcal{P}_m$, where, $\mathcal{P}_0$ consists of just the trivial partition $\{0\}$. A non-trivial partition $\lambda$ may be written as $\lambda = (\lambda_1 \ge \lambda_2 \ge \dots \ge \lambda_l > 0)$, where $|\lambda| = \lambda_1 + \lambda_2 + \dots + \lambda_l$. Set $\langle \lambda , \lambda \rangle = \sum_{j=1}^l (2j-1)\lambda_j$. Following \cite{harv}, we define for each partition $\lambda$ a modified hook polynomial $\mathcal{H}_\lambda(q)$ by
\begin{equation*}
\mathcal{H}_\lambda(q) = q^{-\frac{1}{2}\langle \lambda , \lambda \rangle } \prod_z (1 - q^{h(z)}),
\end{equation*}
where the product is over the boxes $z$ of the Young diagram and $h(z)$ is the hook length of $z$. For the trivial partition $\{0\}$, set $\mathcal{H}_{ \{0\} }(q) = 1$.\\

For each positive integer $r$, let $\Gamma_r = \widehat{\mathbb{F}}_{q^d}^\times$. If $r$ divides $s$ then the norm map $\mathbb{F}_{q^s} \to \mathbb{F}_{q^r}$ induces an inclusion $\Gamma_r \to \Gamma_s$. Let $\Gamma = \varinjlim \Gamma_r$ be the direct limit of the $\Gamma_r$ with respect to these inclusions. The Frobenius automorphism $\sigma$ acts on $\Gamma$ via $\sigma(\gamma) = \gamma^q$. The degree $d(\gamma)$ of $\gamma \in \Gamma$ is defined as the size of the orbit of $\gamma$ under the Frobenius action. There is a canonical bijection between the irreducible characters of $GL_n(\mathbb{F}_q)$ and maps $\Lambda \colon \Gamma \to \mathcal{P}$ which commute with $\sigma$ and satisfy
\begin{equation*}
| \Lambda | = \sum_{\gamma} |\Lambda(\gamma)| = n.
\end{equation*}
In particular, finiteness of the sum means that $\Lambda(\gamma) = \{0\}$ for all but finitely many $\gamma$. Let $\mathcal{P}_n(\Gamma)$ be the set of all maps $\Gamma \to \mathcal{P}$ satisfying these conditions. Given a $\Lambda \in \mathcal{P}_n(\Gamma)$, let $m_{d,\lambda}$ be the multiplicity of $(d,\lambda)$ in $\Lambda$, where $d$ is a positive integer and $\lambda$ a partition. That is,
\begin{equation*}
m_{d,\lambda} = | \{ \gamma \in \Gamma \; | \; d(\gamma) = d, \; \Lambda(\gamma) = \lambda \}|.
\end{equation*}
Note that $|\Lambda| = \sum_{d,\lambda} m_{d,\lambda} d |\lambda|$. The collection $\{ m_{d,\lambda} \}$ of all multiplicities is called the {\em type} of $\Lambda$ and will be denoted by $\tau(\Lambda)$. In particular, $|\Lambda|$ depends only on the type and so we will write $|\tau|$ for $|\Lambda|$ where $\tau = \tau(\Lambda)$. For a given type $\tau$, let
\begin{equation*}
\mathcal{H}_\tau(q) = \prod_{d,\lambda} \mathcal{H}_{\lambda}(q^d)^{m_{d,\lambda}}.
\end{equation*}
Given $\Lambda \in \mathcal{P}_n(\Gamma)$, let $\chi_\Lambda$ be the corresponding irreducible character. Then \cite{harv}
\begin{equation*}
\frac{|GL_n(\mathbb{F}_q)|}{\chi_\Lambda(1)} = (-1)^n q^{\frac{1}{2}n^2} \mathcal{H}_{\tau'},
\end{equation*}
where $\tau' = \tau(\Lambda')$ and $\Lambda'$ is the map $\Gamma \to \mathcal{P}$ sending $\gamma$ to $\Lambda(\gamma)'$ (the dual partition of $\Lambda(\gamma)$).\\

To compute the right hand side of (\ref{equgl3hom}), we simply need to work out which types $\tau$ have $|\tau| = 3$ and then count the number of $\Lambda$ of each type. We summarise this information in the following table:

\begin{equation*}
\renewcommand{\arraystretch}{1.4}
\begin{tabular}{|c|c|c|}
\hline
$\tau$ & $\mathcal{H}_\tau(q)$ & $\#$ of $\Lambda$ of type $\tau$ \\
\hline
$m_{3,1}=1$ & $-q^{-\frac{3}{2}}(q^3-1)$ & $(q^3-q)/3$ \\
$m_{2,1}= m_{1,1} = 1$ & $q^{-\frac{3}{2}}(q^2-1)(q-1)$ & $(q^2-q)(q-1)/2$ \\
$m_{1,3}=1$ & $-q^{-\frac{3}{2}}(q^3-1)(q^2-1)(q-1)$ & $(q-1)$ \\
$m_{1,2+1}=1$ & $-q^{-\frac{5}{2}}(q^3-1)(q-1)^2$ & $(q-1)$ \\
$m_{1,1+1+1}=1$ & $-q^{-\frac{9}{2}}(q^3-1)(q^2-1)(q-1)$ & $(q-1)$ \\
$m_{1,2}=m_{1,1}=1$ & $-q^{-\frac{3}{2}}(q^2-1)(q-1)^2$ & $(q-1)(q-2)$ \\
$m_{1,1+1}=m_{1,1}=1$ & $-q^{-\frac{5}{2}}(q^2-1)(q-1)^2$ & $(q-1)(q-2)$ \\
$m_{1,1}=3$ & $-q^{-\frac{3}{2}}(q-1)^3$ & $(q-1)(q-2)(q-3)/6$ \\
\hline
\end{tabular}
\end{equation*}
\vspace{0.1cm}

\noindent Thus

\begin{equation*}
\begin{aligned}
\frac{|Hom(\Gamma , GL_3(\mathbb{F}_q))|}{|GL_3(\mathbb{F}_q)|} &= \frac{(q^3-q)}{3} \left( q^3(q^3-1) \right)^{2g-2} + \frac{(q^2-q)(q-1)}{2} \left( q^3(q^2-1)(q-1) \right)^{2g-2} \\
& \; \; \; \; \; + (q-1) \left( q^3(q^3-1)(q^2-1)(q-1) \right)^{2g-2} + (q-1)\left( q^2(q^3-1)(q-1)^2 \right)^{2g-2} \\
& \; \; \; \; \; + (q-1)\left( (q^3-1)(q^2-1)(q-1) \right)^{2g-2} + (q-1)(q-2)\left( q^3(q^2-1)(q-1)^2 \right)^{2g-2} \\
& \; \; \; \; \; + (q-1)(q-2)\left( q^2(q^2-1)(q-1)^2 \right)^{2g-2} + \frac{(q-1)(q-2)(q-3)}{6}\left( q^3(q-1)^3 \right)^{2g-2}.
\end{aligned}
\end{equation*}
We can simplify this further, giving:

\begin{align}
\frac{|Hom(\Gamma , GL_3(\mathbb{F}_q))|}{(q-1)^{2g-1}|GL_3(\mathbb{F}_q)|} &= \frac{(q^2+q)}{3} \left( q^5+q^4+q^3 \right)^{2g-2} + \frac{(q^2-q)}{2} \left( q^5-q^3 \right)^{2g-2} \nonumber \\
& \; \; \; \; \; + \left( q^8-q^6-q^5+q^3 \right)^{2g-2} + \left(  q^6-q^5-q^3+q^2 \right)^{2g-2} \nonumber \\
& \; \; \; \; \; + \left( q^5-q^3-q^2+1 \right)^{2g-2} + (q-2)\left( q^6-q^5-q^4+q^3 \right)^{2g-2} \label{equhomintogl3} \\
& \; \; \; \; \; + (q-2)\left( q^5-q^4-q^3+q^2 \right)^{2g-2} + \frac{(q-2)(q-3)}{6}\left( q^5-2q^4+q^3 \right)^{2g-2}. \nonumber
\end{align}

\subsection{Case of $SL_3(\mathbb{F}_q)$}\label{sechomssl3}

We assume $q = 1 \; ({\rm mod} \; 3)$ and compute $|Hom(\Gamma , SL_3(\mathbb{F}_q))|$, again for the case that $\Gamma = \pi_1(\Sigma_g)$ is the fundamental group of a compact oriented surface of genus $g$. For this, we need to relate the irreducible characters of $SL_n(\mathbb{F}_q)$ with $GL_n(\mathbb{F}_q)$. Let $Irr(GL_n)$ be the set of irreducible characters for $GL_n(\mathbb{F}_q)$ and $Irr(SL_n)$ the irreducible characters of $SL_n(\mathbb{F}_q)$. There is an action of $\mathbb{F}_q^\times$ on $Irr(SL_n)$ by $(g\theta)(h) = \theta(\tilde{g} h \tilde{g}^{-1} )$, where $\tilde{g}$ is any element of $GL_n(\mathbb{F}_q)$ with $det(\tilde{g}) = g$. There is also an action of $\widehat{\mathbb{F}}_q^\times$ on $Irr(GL_n)$ by tensor product, where one views $\widehat{\mathbb{F}}_q^\times$ as the $1$-dimensional representations of $GL_n(\mathbb{F}_q)$ which factor through the determinant $det \colon GL_n(\mathbb{F}_q) \to \mathbb{F}_q^\times$.

\begin{proposition}[\cite{leh},\cite{kagr}]\label{propirrslgl}
There is a bijection $Irr(GL_n)/ \widehat{\mathbb{F}}_q^\times \cong Irr(SL_n)/\mathbb{F}_q^\times$ given as follows. Let $\chi \in Irr(GL_n)$. Then there is a unique $\mathbb{F}_q^\times$-orbit $[\theta] \in Irr(SL_n) / \mathbb{F}_q^\times$ such that
\begin{equation*}
\chi|_{SL_n(\mathbb{F}_q)} = \sum_{\theta' \in [\theta] } \theta'.
\end{equation*}
Moreover, letting $Stab(\chi) \subseteq \widehat{\mathbb{F}}_q^\times$ denote the stabiliser of $\chi$ under the $\widehat{\mathbb{F}}_q^\times$-action we have that $| [\theta] | = |Stab(\chi)|$.
\end{proposition}
For a given $\theta \in Irr(SL_n)$ let $t(\theta)$ be the size of the orbit $|[\theta]|$. If $\chi \in Irr(GL_n)$ is related to $\theta$ as in Proposition \ref{propirrslgl}, then we will also denote $t(\theta)$ by $t(\chi)$. The action of $\mathbb{F}_q^\times$ on $Irr(SL_n)$ preserves dimension of the representation, i.e. $(g\theta)(1) = \theta(1)$ for all $g,\theta$. It follows that $\chi(1) = t(\chi)\theta(1)$. Using this, we find:
\begin{equation*}
\begin{aligned}
\frac{|Hom(\Gamma , SL_3(\mathbb{F}_q))|}{|SL_3(\mathbb{F}_q)|} &= \sum_{\theta \in Irr(SL_n)} \left( \frac{|SL_3(\mathbb{F}_q)|}{\theta(1)} \right)^{2g-2} \\
&= \sum_{[\theta] \in Irr(SL_n)/\mathbb{F}_q^\times} t(\theta) \left( \frac{|SL_3(\mathbb{F}_q)|}{\theta(1)} \right)^{2g-2} \\
&= \sum_{[\chi] \in Irr(GL_n)/\widehat{\mathbb{F}}_q^\times} t(\chi) \left( \frac{|SL_3(\mathbb{F}_q)|}{\chi(1)/t(\chi)} \right)^{2g-2} \\
&= \sum_{[\chi] \in Irr(GL_n)/\widehat{\mathbb{F}}_q^\times} \frac{t(\chi)^{2g-1}}{(q-1)^{2g-2}} \left( \frac{|GL_3(\mathbb{F}_q)|}{\chi(1)} \right)^{2g-2} \\
&= \sum_{\chi \in Irr(GL_n)} \frac{t(\chi)^{2g}}{(q-1)^{2g-1}} \left( \frac{|GL_3(\mathbb{F}_q)|}{\chi(1)} \right)^{2g-2} \\
&= \sum_{\Lambda} \frac{t(\chi_\Lambda)^{2g}}{(q-1)^{2g-1}} \mathcal{H}_{\tau'}^{2g-2}.
\end{aligned}
\end{equation*}
To complete the computation, we just need to know for each type $\tau$ and for each positive integer $t$, the number of $\Lambda$ such that $\tau(\Lambda) = \tau$ and $t(\chi_\Lambda) = t$. View $\widehat{\mathbb{F}}_q^\times$ as a subgroup of $\Gamma$, so that $\widehat{\mathbb{F}}_q^\times$ acts on $\Gamma$ by translation. Then the action of $\widehat{\mathbb{F}}_q^\times$ on $\Lambda \colon \Gamma \to \mathcal{P}$ is given by precomposition by the action of $\widehat{\mathbb{F}}_q^\times$ on $\Gamma$. From this it is straightforward to determine the $\Lambda \in \mathcal{P}_3(\Gamma)$ for which $t( \chi_\Lambda) \neq 1$. Namely, this can only happen for the types $m_{3,1} = 1$ and $m_{1,1}=3$. In fact, there are $2(q-1)/3$ different $\Lambda$ of type $m_{3,1}=1$ for which $t = 3$ (all others have $t=1$) and there are $(q-1)/3$ different $\Lambda$ of type $m_{1,1}=3$ for which $t = 3$ (all others have $t=1$). Thus we find:
\begin{align}
\frac{|Hom(\Gamma , SL_3(\mathbb{F}_q))|}{|SL_3(\mathbb{F}_q)|} &= \left( 2 \cdot 3^{2g-1} + \frac{(q-1)(q+2)}{3}\right) \left( q^5+q^4+q^3 \right)^{2g-2} + \frac{(q^2-q)}{2} \left( q^5-q^3 \right)^{2g-2} \nonumber \\
& \; \; \; \; \; + \left( q^8-q^6-q^5+q^3 \right)^{2g-2} + \left(  q^6-q^5-q^3+q^2 \right)^{2g-2} \nonumber \\
& \; \; \; \; \; + \left( q^5-q^3-q^2+1 \right)^{2g-2} + (q-2)\left( q^6-q^5-q^4+q^3 \right)^{2g-2} \label{equhomintosl3} \\
& \; \; \; \; \; + (q-2)\left( q^5-q^4-q^3+q^2 \right)^{2g-2} + \left( 3^{2g-1} + \frac{(q-1)(q-4)}{6} \right) \left( q^5-2q^4+q^3 \right)^{2g-2}. \nonumber
\end{align}

\section{Computations for $GL_2$ and $SL_2$ character varieties}\label{seccomputgl2sl2}

\subsection{Free groups}
In this section $\Gamma = F_r$ is the free group on $r$ generators.

\begin{lemma}
Let $A$ be a non-trivial rank $1$ representation of $\Gamma$ over $\mathbb{F}_q$. Then $b^1_A = r-1$.
\end{lemma}
\begin{proof}
Since $F_r$ has cohomological dimension one, equating Euler characteristics gives:
\begin{equation*}
\begin{aligned}
{\rm dim}( H^0(\Gamma , A) ) - {\rm dim}( H^1(\Gamma , A) ) &= {\rm dim}( H^0(\Gamma , \mathbb{F}_q) ) - {\rm dim}( H^1(\Gamma , \mathbb{F}_q) ) \\
& = 1 - r.
\end{aligned}
\end{equation*}
Since $A$ is non-trivial, we have $H^0(\Gamma , A) = 0$ and so ${\rm dim}( H^1(\Gamma , A) ) = r-1$.
\end{proof}

\subsection{Case of $GL_2$}
By Theorem \ref{thmcountingredrank2}, we have:
\begin{equation*}
\begin{aligned}
A_{GL_2}(q) &= (q-1)^r(q^3-q)^{r-1} + (q-1)^r \frac{q}{2} (q+1)^{r-1} + (q-1)^r \frac{(q-2)}{2} (q-1)^{r-1} \\
& \; \; \; \; \; -(q-1)^r \frac{(q^{r-1}-1)}{(q-1)}  - (q-1)^r \left( (q-1)^r - 1 \right) \frac{(q^{r-1} - 1)}{(q-1)} \\
&= (q-1)^r(q^3-q)^{r-1} + (q-1)^r \frac{q}{2} (q+1)^{r-1} + (q-1)^r \frac{(q-2)}{2} (q-1)^{r-1} \\
& \; \; \; \; \; -(q-1)^r(q-1)^{r-1}(q^{r-1} - 1) \\
&= (q-1)^r \left(  (q^3-q)^{r-1} -(q^2-q)^{r-1} + q\left( \frac{1}{2}(q+1)^{r-1} + \frac{1}{2}(q-1)^{r-1}\right) \right).
\end{aligned}
\end{equation*}
This agrees with \cite{more}. This is polynomial count and hence gives the $E$-polynomial of the corresponding complex character variety.

\subsection{Case of $SL_2$}

By Theorem \ref{thmcountingredrank2}, we have:
\begin{equation*}
\begin{aligned}
A_{SL_2}(q) &= (q^3-q)^{r-1} + \left( 1 - \frac{1}{q+1} \right) \frac{(q+1)^r}{2} + \left( 1 - \frac{1}{q-1} \right) \frac{(q-1)^r}{2} \\
& \; \; \; \; \; -2^r[r-1]_q - \sum_{ \{ A | A^2 \neq 1 \} } [r-1]_q \\
&= (q^3-q)^{r-1} + \frac{q}{2}(q+1)^{r-1} + \frac{q-2}{2}(q-1)^{r-1} - 2^r [r-1]_q \\
& \; \; \; \; \; -\left( (q-1)^r - 2^r \right) [r-1]_q \\
&= (q^3-q)^{r-1} + \frac{q}{2}(q+1)^{r-1} + \frac{q}{2}(q-1)^{r-1} - (q-1)^{r-1} \\
& \; \; \; \; \; - (q-1)^{r-1}( q^{r-1} - 1)\\
&= (q^3-q)^{r-1} + \frac{q}{2}(q+1)^{r-1} + \frac{q}{2}(q-1)^{r-1} - q^{r-1}(q-1)^{r-1} \\
&= (q^3-q)^{r-1} - (q^2-q)^{r-1} + q\left( \frac{1}{2}(q+1)^{r-1} + \frac{1}{2}(q-1)^{r-1} \right).
\end{aligned}
\end{equation*}
This is polynomial count and agrees with \cite{more}.

\subsection{Fundamental groups of compact orientable surfaces}

In this section $\Gamma = \pi_1(\Sigma_g)$, where $\Sigma_g$ is a compact oriented surface of genus $g \ge 1$.

\subsection{Case of $GL_2$}

Assume that $q$ is odd. By Theorem \ref{thmcountingredrank2}, we have:

\begin{equation*}
\begin{aligned}
A_{GL_2}(q) &= \frac{ | Hom(\Gamma , GL_2(\mathbb{F}_q)) |}{| PGL_2(\mathbb{F}_q)|} + \left( 1 - \frac{1}{q+1} \right) \frac{(q^2-1)^{2g}}{2} + \left( 1 - \frac{1}{q-1} \right) \frac{(q-1)^{4g}}{2} \\
& \; \; \; \; \; -(q-1)^{2g}[2g-1]_q  - (q-1)^{2g}\left( (q-1)^{2g} - 1 \right) [2g-2]_q \\
&= \frac{ | Hom(\Gamma , GL_2(\mathbb{F}_q)) |}{| PGL_2(\mathbb{F}_q)|} + \frac{q}{2}(q+1)^{2g-1}(q-1)^{2g} + \frac{(q-2)}{2}(q-1)^{4g-1} \\
& \; \; \; \; \; -(q-1)^{4g-1}(q^{2g-2}-1) - (q-1)^{2g}q^{2g-2} \\
&= \frac{ | Hom(\Gamma , GL_2(\mathbb{F}_q)) |}{| PGL_2(\mathbb{F}_q)|} + \frac{q}{2}(q+1)^{2g-1}(q-1)^{2g} + \frac{q}{2}(q-1)^{4g-1} \\
& \; \; \; \; \; -(q-1)^{4g-1}q^{2g-2} - (q-1)^{2g}q^{2g-2} \\
&= \frac{ | Hom(\Gamma , GL_2(\mathbb{F}_q)) |}{| PGL_2(\mathbb{F}_q)|} \\
& \; \; \; \; \; +(q-1)^{2g} \left( q\left( \frac{1}{2}(q+1)^{2g-1} + (q-1)^{2g-1} \right) -(q-1)^{2g-1}q^{2g-2} - q^{2g-2} \right).
\end{aligned}
\end{equation*}

By Equation (\ref{equhomcos}) and the character table for $GL_2(\mathbb{F}_q)$, we also have:

\begin{equation*}
\begin{aligned}
\frac{ | Hom(\Gamma , GL_2(\mathbb{F}_q)) |}{| GL_2(\mathbb{F}_q)|} &= \sum_{\chi} \left( \frac{|GL_2(\mathbb{F}_q)|}{\chi(1)} \right)^{2g-2} \\
&= \frac{(q-1)(q-2)}{2}(q^2-q)^{2g-2}(q-1)^{2g-2} + \frac{q(q-1)}{2}(q^2+q)^{2g-2}(q-1)^{2g-2} \\
& \; \; \; \; \; + (q-1)^{2g-1}(q^3-q)^{2g-2} + (q-1)^{2g-1}(q^2-1)^{2g-2}.
\end{aligned}
\end{equation*}
Hence:

\begin{equation}\label{equhomintogl2}
\begin{aligned}
\frac{ | Hom(\Gamma , GL_2(\mathbb{F}_q)) |}{| PGL_2(\mathbb{F}_q)|} &= \frac{(q-2)}{2}(q^2-q)^{2g-2}(q-1)^{2g} + \frac{q}{2}(q^2+q)^{2g-2}(q-1)^{2g} \\
& \; \; \; \; \; + (q^3-q)^{2g-2}(q-1)^{2g} + (q^2-1)^{2g-2}(q-1)^{2g}.
\end{aligned}
\end{equation}
This gives:

\begin{equation*}
\begin{aligned}
A_{GL_2}(q) &= (q-1)^{2g} \left( (q^3-q)^{2g-2} + (q^2-1)^{2g-2} + q\left( \frac{1}{2}(q^2+q)^{2g-2} + \frac{1}{2}(q^2-q)^{2g-2} \right) \right. \\
& \; \; \; \; \; \left. -q(q^2-q)^{2g-2} -q^{2g-2} + q\left( \frac{1}{2}(q+1)^{2g-1} + \frac{1}{2}(q-1)^{2g-1} \right) \right).
\end{aligned}
\end{equation*}

\noindent This is polynomial count. Note that it is $(q-1)^{2g}$ times a polynomial in $q$ but this polynomial is not the corresponding $E$-polynomial for $SL_2$.

\subsection{Case of $SL_2$}

We again assume that $q$ is odd. Then by Theorem \ref{thmcountingredrank2}, we get:

\begin{equation*}
\begin{aligned}
A_{SL_2}(q) &= \frac{ | Hom(\Gamma , SL_2(\mathbb{F}_q)) |}{|PGL_2(\mathbb{F}_q)|} + \left( 1 - \frac{1}{q+1} \right) \frac{(q+1)^{2g}}{2} + \left( 1 - \frac{1}{q-1} \right) \frac{(q-1)^{2g}}{2} \\
& \; \; \; \; \; -2^{2g}[2g-1]_q  -\left( (q-1)^{2g} - 2^{2g} \right) [2g-2]_q \\
&= \frac{ | Hom(\Gamma , SL_2(\mathbb{F}_q)) |}{|PGL_2(\mathbb{F}_q)|} + \frac{q}{2}(q+1)^{2g-1} + \frac{q-2}{2}(q-1)^{2g-1} \\
& \; \; \; \; \; - (q-1)^{2g-1}(q^{2g-2}-1) -2^{2g}q^{2g-2} \\
&= \frac{ | Hom(\Gamma , SL_2(\mathbb{F}_q)) |}{|PGL_2(\mathbb{F}_q)|} + q\left( \frac{1}{2}(q+1)^{2g-1} + \frac{1}{2}(q-1)^{2g-1} \right) -q^{2g-2}\left( (q-1)^{2g-1} + 2^{2g} \right) \\
&= \frac{ | Hom(\Gamma , SL_2(\mathbb{F}_q)) |}{|PGL_2(\mathbb{F}_q)|} + q\left( \frac{1}{2}(q+1)^{2g-1} + \frac{1}{2}(q-1)^{2g-1} \right) -(q-1)(q^2-q)^{2g-2} -2^{2g}q^{2g-2}.
\end{aligned}
\end{equation*}
Using Equation \ref{equhomcos} and the character table for $SL_2(\mathbb{F}_q)$, we have:

\begin{equation*}
\begin{aligned}
\frac{ | Hom(\Gamma , SL_2(\mathbb{F}_q)) |}{|SL_2(\mathbb{F}_q)|} &= \sum_{\chi} \left( \frac{|SL_2(\mathbb{F}_q)|}{\chi(1)} \right)^{2g-2} \\
&= \frac{(q-3)}{2}(q^2-q)^{2g-2} + 2^{2g-1} (q^2-q)^{2g-2} + \frac{(q-1)}{2}( q^2+q)^{2g-2} \\
& \; \; \; \; \; +2^{2g-1} (q^2+q)^{2g-2} + (q^3-q)^{2g-2} + (q^2-1)^{2g-2},
\end{aligned}
\end{equation*}
so in total we get:

\begin{equation*}
\begin{aligned}
A_{SL_2}(q) &= (q^3-q)^{2g-2} + (q^2-1)^{2g-2} - q(q^2-q)^{2g-2} - 2^{2g}q^{2g-2} \\
& \; \; \; \; \; + \frac{(q-1)}{2}\left(  (q^2+q)^{2g-2} + (q^2-q)^{2g-2} \right) + 2^{2g-1}\left(  (q^2+q)^{2g-2} + (q^2-q)^{2g-2} \right) \\
& \; \; \; \; \; + q\left( \frac{1}{2}(q+1)^{2g-1} + \frac{1}{2}(q-1)^{2g-1} \right).
\end{aligned}
\end{equation*}

\noindent This calculation recovers at once the results of \cite{lomune}, which computed the cases $g=1,2$, \cite{mamu1} which computed the case $g=3$ and \cite{mamu2} which computed the case $g > 3$.
\vspace{0.1cm}

\subsection{Fundamental groups of compact non-orientable surfaces}

In this section $\Gamma = \pi_1( \Sigma_k)$, where $\Sigma_k$ is a connected sum of $k \ge 1$ copies of $\mathbb{RP}^2$. Recall that $\Gamma$ has the presentation

\begin{equation*}
\Gamma = \langle a_1 , a_2 \dots , a_k \; | \; a_1^2 a_2^2 \dots a_k^2 = 1 \rangle.
\end{equation*}

\noindent Let $A_{\rm orn}$ denote the representation $A_{\rm orn} \colon \Gamma \to \mathbb{F}_q^\times$ sending $a_1, a_2 , \dots , a_k$ to $-1 \in \mathbb{F}_q^\times$.

\begin{lemma}
Let $q = 1 \; ({\rm mod} \; 4)$. Then there exists $L \colon \Gamma \to \mathbb{F}_q^\times$ such that $L^2 = A_{\rm orn}$ if and only if $k$ is even. When $k$ is even there exists exactly $m_2 = 2^k$ such $L$.
\end{lemma}
\begin{proof}
Suppose $L \colon \Gamma \to \mathbb{F}_q^\times$ is a square root of $A_{\rm orn}$. So for each $i$, $L$ sends $a_i$ to a square root of $-1$. Thus $L$ sends $a_1^2 a_2^2 \dots a_k^2$ to $(-1)^k$. As this equals $1$, we must have that $k$ is even. Since $q = 1 \;({\rm mod} \; 4)$, we have that $-1$ has two square roots $\mathbb{F}_q^\times$. It follows that when $k$ is even, we get $2^k$ square roots of $A_{\rm orn}$.
\end{proof}

\begin{lemma}
Let $A \colon \Gamma \to \mathbb{F}_q^\times$ be different from $1$ and $A_{\rm orn}$. Then ${\rm dim}( H^1(\Gamma , A) ) = k-2$. We also have that ${\rm dim}( H^1(\Gamma , A_{\rm orn}) ) = k-1$.
\end{lemma}
\begin{proof}
Follows from Poincar\'e duality and the fact that $\Sigma_k$ has Euler characteristic $2-k$.
\end{proof}

\subsection{Case of $GL_2$}

Assume $q = 1 \; ({\rm mod} \; 4)$. By Theorem \ref{thmcountingredrank2}, we have:
\begin{equation*}
\begin{aligned}
A_{GL_2}(q) &= \frac{ | Hom(\Gamma , GL_2(\mathbb{F}_q)) |}{| PGL_2(\mathbb{F}_q)|} + q(q-1)^{k-1}(q+1)^{k-2} + 2(q-2)(q-1)^{2k-3} \\
& \; \; \; \; \; -2(q-1)^{k-1}[k-2]_q -2(q-1)^{k-1} 
\left( 2(q-1)^{k-1} - 1 \right) [k-2]_q
-2(q-1)^{k-1}q^{k-2} \\
&= \frac{ | Hom(\Gamma , GL_2(\mathbb{F}_q)) |}{| PGL_2(\mathbb{F}_q)|} + q(q-1)^{k-1}(q+1)^{k-2} + 2q(q-1)^{2k-3} -4(q-1)^{2k-3}\\
& \; \; \; \; \; -4(q-1)^{2k-3}(q^{k-2} -1) -2(q-1)^{k-1}q^{k-2} \\
&= \frac{ | Hom(\Gamma , GL_2(\mathbb{F}_q)) |}{| PGL_2(\mathbb{F}_q)|} +
q(q-1)^{k-1}\left( (q+1)^{k-2} + 2(q-1)^{k-2} \right) \\
& \; \; \; \; \; -4(q-1)^{2k-3}q^{k-2} -2(q-1)^{k-1}q^{k-2}.
\end{aligned}
\end{equation*}
By Equation (\ref{equhomcnos}), the character table and Frobenius-Schur indicators for $GL_2(\mathbb{F}_q)$, we have:

\begin{equation*}
\begin{aligned}
\frac{ | Hom(\Gamma , GL_2(\mathbb{F}_q)) |}{(q-1)| PGL_2(\mathbb{F}_q)|} &= \sum_{\chi} \left( \frac{|GL_2(\mathbb{F}_q)|}{\chi(1)} \right)^{k-2} \nu_2^k(\chi) \\
&= \frac{(q-1)}{2}q^{k-2}(q-1)^{2k-4} + \frac{(q-1)}{2} (q^3-q)^{k-2} \\
& \; \; \; \; \; + 2q^{k-2}(q+1)^{k-2}(q-1)^{2k-4} + 2(q+1)^{k-2}(q-1)^{2k-4} \\
& \! \! \! \! \! \! \! \! \! \! \! \! \! \! \! \! \! = (q-1)^{k-2}\left( \frac{(q-1)}{2}(q^2-q)^{k-2} +\frac{(q-1)}{2}(q^2+q)^{k-2} + 2(q^3-q)^{k-2} + 2(q^2-1)^{k-2} \right).
\end{aligned}
\end{equation*}
Thus:
\begin{equation*}
\begin{aligned}
\frac{A_{GL_2}(q)}{(q-1)^{k-1}} &= \frac{(q-1)}{2}(q^2-q)^{k-2} +\frac{(q-1)}{2}(q^2+q)^{k-2} + 2(q^3-q)^{k-2} + 2(q^2-1)^{k-2} \\
& \; \; \; \; \; + q\left( (q+1)^{k-2} + 2(q-1)^{k-2} \right) -4(q-1)^{k-2}q^{k-2} -2q^{k-2}.
\end{aligned}
\end{equation*}
This is polynomial count. When $k >3$, the leading order term in $A_{GL_2}(q)$ is $2q^{4k-7}$, so this character variety has two irreducible components of dimension $4k-7$. Let $\rho \colon \Gamma \to GL_2(\mathbb{F}_q)$ be a representation. Then $det( \rho( a_1^2 \dots a_k^2) ) = 1$, so $det( \rho (a_1 \dots a_k ) ) = \pm 1$. The two components are distinguished by the value of $det( \rho( a_1 \dots a_k ))$, i.e. whether the determinant line is a square or not.

\subsection{Case of $SL_2$}

We again assume $q = 1 \; ({\rm mod} \; 4)$. First suppose $k$ is even. By Theorem \ref{thmcountingredrank2}, we have:
\begin{equation*}
\begin{aligned}
A_{SL_2}(q) &= \frac{ | Hom(\Gamma , SL_2(\mathbb{F}_q)) |}{|PGL_2(\mathbb{F}_q)|} + q(q+1)^{k-2} + (q-2)(q-1)^{k-2} \\
& \; \; \; \; \; -2^k [k-2]_q  - \left( 2(q-1)^{k-1} - 2^{k+1} \right) [k-2]_q - 2^k [k-1]_q \\
&= \frac{ | Hom(\Gamma , SL_2(\mathbb{F}_q)) |}{|PGL_2(\mathbb{F}_q)|} + q(q+1)^{k-2} + q(q-1)^{k-2} -2(q-1)^{k-2} \\
& \; \; \; \; \; +2^k [k-2]_q  - 2 (q-1)^{k-2}(q^{k-2} - 1) - 2^k [k-1]_q \\
&= \frac{ | Hom(\Gamma , SL_2(\mathbb{F}_q)) |}{|PGL_2(\mathbb{F}_q)|} + q(q+1)^{k-2} + q(q-1)^{k-2} -2^k q^{k-2} - 2(q^2-q)^{k-2}.
\end{aligned}
\end{equation*}
Second, suppose $k$ is odd. Then:
\begin{equation*}
\begin{aligned}
A_{SL_2}(q) &= \frac{ | Hom(\Gamma , SL_2(\mathbb{F}_q)) |}{|SL_2(\mathbb{F}_q)|} + q(q+1)^{k-2} + (q-2)(q-1)^{k-2} \\
& \; \; \; \; \; -2^k [k-2]_q  - \left( 2(q-1)^{k-1} - 2^{k} \right) [k-2]_q \\
&=\frac{ | Hom(\Gamma , SL_2(\mathbb{F}_q)) |}{|SL_2(\mathbb{F}_q)|} + q(q+1)^{k-2} + (q-2)(q-1)^{k-2} \\
& \; \; \; \; \; - 2(q-1)^{k-2}(q^{k-2} - 1) \\
&=\frac{ | Hom(\Gamma , SL_2(\mathbb{F}_q)) |}{|SL_2(\mathbb{F}_q)|} + q(q+1)^{k-2} + q(q-1)^{k-2} -2(q^2-q)^{k-2}.
\end{aligned}
\end{equation*}
By Equation (\ref{equhomcnos}), the character table and Frobenius-Schur indicators for $SL_2(\mathbb{F}_q)$, we find, for $k$ even, that:
\begin{equation*}
\begin{aligned}
 \frac{ | Hom(\Gamma , SL_2(\mathbb{F}_q)) |}{|SL_2(\mathbb{F}_q)|} &= \sum_{\chi} \left( \frac{|SL_2(\mathbb{F}_q)|}{\chi(1)} \right) \nu_2(\chi)^k \\
&= \frac{(q-3)}{2}(q^2-q)^{k-2} + 2^{k-1}(q^2-q)^{k-2} + \frac{(q-1)}{2}(q^2+q)^{k-2} \\
& \; \; \; \; \; +2^{k-1}(q^2+q)^{k-2} + (q^3-q)^{k-2} + (q^2-1)^{k-2} \\
&= \frac{(q-1)}{2}\left( (q^2+q)^{k-2} + (q^2-q)^{k-2} \right) + (2^{k-1}-1)(q^2-q)^{k-2} \\
& \; \; \; \; \; +2^{k-1}(q^2+q) + (q^3-q)^{k-2} + (q^2-1)^{k-2}.
\end{aligned}
\end{equation*}
For $k$ odd, we find:
\begin{equation*}
\begin{aligned}
 \frac{ | Hom(\Gamma , SL_2(\mathbb{F}_q)) |}{|SL_2(\mathbb{F}_q)|} &= \sum_{\chi} \left( \frac{|SL_2(\mathbb{F}_q)|}{\chi(1)} \right) \nu_2(\chi)^k \\
&= -(q^2-q)^{k-2} + 2^{k-1}(q^2-q)^{k-2} \\
& \; \; \; \; \; -2^{k-1}(q^2+q)^{k-2} + (q^3-q)^{k-2} + (q^2-1)^{k-2}\\
&= (q^3-q)^{k-2} + (q^2-1)^{k-2} -2^{k-1}(q^2+q)^{k-2} + (2^{k-1}-1)(q^2-q)^{k-2}.
\end{aligned}
\end{equation*}
So when $k$ is even, we have:

\begin{equation*}
\begin{aligned}
A_{SL_2}(q) &= \frac{(q-1)}{2}\left( (q^2+q)^{k-2} + (q^2-q)^{k-2} \right) + (2^{k-1}-1)(q^2-q)^{k-2} \\
& \; \; \; \; \; +2^{k-1}(q^2+q) + (q^3-q)^{k-2} + (q^2-1)^{k-2} \\
& \; \; \; \; \; +q(q+1)^{k-2} + q(q-1)^{k-2} -2^k q^{k-2} - 2(q^2-q)^{k-2} \\
&= \frac{(q-1)}{2}\left( (q^2+q)^{k-2} + (q^2-q)^{k-2} \right) + (2^{k-1}-3)(q^2-q)^{k-2} \\
& \; \; \; \; \; +2^{k-1}(q^2+q) + (q^3-q)^{k-2} + (q^2-1)^{k-2} \\
& \; \; \; \; \; +q(q+1)^{k-2} + q(q-1)^{k-2} -2^k q^{k-2} \\
&= (q^3-q)^{k-2} + (q^2-1)^{k-2} + \left( \frac{(q-1)}{2} + 2^{k-1} \right)\left( (q^2+q)^{k-2} + (q^2-q)^{k-2} \right) \\
& \; \; \;\; \; -3(q^2-q)^{k-2} - 2^k q^{k-2} +q\left( (q+1)^{k-2} + (q-1)^{k-2} \right).
\end{aligned}
\end{equation*}

\noindent For example, putting $k=2$, $\Sigma_k$ is the Klein bottle and we get:

\begin{equation*}
\begin{aligned}
A_{SL_2}(q) &= 1 + 1 + ( q+3) -3 - 4 + 2q \\
&= 3q-2.
\end{aligned}
\end{equation*}
When $k$ is odd, we have
\begin{equation*}
\begin{aligned}
A_{SL_2}(q) &= (q^3-q)^{k-2} + (q^2-1)^{k-2} -2^{k-1}(q^2+q)^{k-2} + (2^{k-1}-1)(q^2-q)^{k-2} \\
& \; \; \; \; \;  + q(q+1)^{k-2} + q(q-1)^{k-2} -2(q^2-q)^{k-2} \\
&= (q^3-q)^{k-2} + (q^2-1)^{k-2} -2^{k-1}(q^2+q)^{k-2} + (2^{k-1}-3)(q^2-q)^{k-2} \\
& \; \; \; \; \;  + q \left( (q+1)^{k-2} + (q-1)^{k-2} \right).
\end{aligned}
\end{equation*}
For example, when $k = 3$, so $\Sigma_k$ is a connected sum of three copies of $\mathbb{RP}^2$, we get:
\begin{equation*}
\begin{aligned}
A_{SL_2}(q) &= (q^3-q) + (q^2-1) - 4(q^2+q) + (q^2 -q) + 2q^2 \\
&= q^3 - 6q - 1.
\end{aligned}
\end{equation*}
These are polynomial count. The $E$-polynomials of the $SL_2(\mathbb{C})$-character varieties of $\pi_1(\Sigma_k)$ were computed for $k=2,3$ in \cite{mart}. The above examples show that our general formula agrees with \cite{mart} in these cases.\\

\noindent Setting $q=1$, we obtain:

\begin{theorem}
For $k>2$, the Euler characteristic of $Rep( \pi_1(\Sigma_k) , SL_2(\mathbb{C}))$ is $2^{2k-3} -3 \cdot 2^{k-2}$ if $k$ is even, $-2^{2k-3} + 2^{k-2}$ if $k$ is odd.
\end{theorem}
\vspace{0.2cm}

\subsection{Torus knot groups}

Let $a,b$ be coprime integers and $\Gamma = \langle x , y \; | \; x^a = y^b \rangle$ be the fundamental group of the complement in $S^3$ of an $(a,b)$-torus knot. Throughout we assume $q = 1 \; ({\rm mod} \; ab)$ and $q = 1 \; ({\rm mod} \; 4)$.
\vspace{0.1cm}

\begin{lemma}
Let $A \colon \Gamma \to \mathbb{F}_q^\times$. Then there exists a unique $w \in \mathbb{F}_q^\times$ such that $A(x) = w^b$, $A(y) = w^a$. In particular $m_{q-1} = q-1$.
\end{lemma}
\begin{proof}
Recall that $\mathbb{F}_q^\times$ is cyclic of order $q-1$. Consider the map $f \colon \mathbb{Z}_{q-1} \oplus \mathbb{Z}_{q-1} \to \mathbb{Z}_{q-1}$ given by $f(u,v) = au - bv$. Since $a,b$ are coprime there exists $u,v$ such that $au-bv = 1$. Hence $f$ is surjective and so the kernel of $f$ has $q-1$ elements. Consider the map $g \colon \mathbb{Z}_{q-1} \to \mathbb{Z}_{q-1} \oplus \mathbb{Z}_{q-1}$ given by $g(w) = (bw,aw)$. Clearly $g$ maps to the kernel of $f$. Moreover $g$ is injective, for if $bw = aw = 0$, then $w = (au-bv)w = u(aw) - v(bw) = 0$. Thus every solution to $au=bv$ is of the form $u = bw$, $v = aw$.
\end{proof}

\begin{lemma}
Let $A \colon \Gamma \to \mathbb{F}_q^\times$. Then
\begin{equation*}
b^1_A = \begin{cases} 1 & \text{if } A = 1, \\ 
1 & \text{if } A(x^a) = A(y^b) = 1, \; \; A(x) \neq 1, \; \; A(y) \neq 1, \\
0 & \text{otherwise}. \end{cases}
\end{equation*}
\end{lemma}
\begin{proof}
Let $A$ be the representation $A(x) = w^b$, $A(y) = w^a$. Consider the cocycles for $H^1( \Gamma , A)$. Such a cocycle is determined by its values $\alpha, \beta$ on $x$ and $y$. The equation $x^a = y^b$ gives the cocycle condition
\begin{equation*}
(1 + w^b + w^{2b} + \dots + w^{ab-b})\alpha = (1+w^a + w^{2a} + \dots + w^{ab-a})\beta.
\end{equation*}
A coboundary is a solution to the cocycle condition of the form $\alpha = (w^b-1)z$, $\beta = (w^a-1)z$, for some $z \in \mathbb{F}_q$. The result now follows easily.
\end{proof}

\subsection{Case of $GL_2$}

Assume first that $a$ and $b$ are odd. By Equation (\ref{equhomtk}) and the Frobenius-Schur indicators for $GL_2(\mathbb{F}_q)$, we get:

\begin{equation*}
\begin{aligned}
\frac{|Hom(\Gamma , GL_2(\mathbb{F}_q))|}{|GL_2(\mathbb{F}_q)|} &= \sum_{\chi} \nu_a(\chi) \nu_b(\chi) \\
&= \frac{1}{2} \sum_{\alpha \neq \beta} \left( \delta_{\alpha^a} \delta_{\beta^a} + \frac{(a-1)}{2}\delta_{\alpha^a \beta^a} \right)
\left( \delta_{\alpha^b} \delta_{\beta^b} + \frac{(b-1)}{2}\delta_{\alpha^b \beta^b} \right) + 1 \\
& \; \; \; \; \; + \frac{(a-1)(b-1)}{2} + \frac{1}{2} \sum_{\omega \neq \omega^q} \left( -\delta_{\omega^a} + \frac{(a-1)}{2}\delta_{\omega^{(q+1)a}} \right)
\left( -\delta_{\omega^b} + \frac{(b-1)}{2}\delta_{\omega^{(q+1)b}} \right)\\
&= 1 + (a-1)(b-1)\frac{(q+2)}{4}.
\end{aligned}
\end{equation*}

\noindent Hence by Theorem \ref{thmcountingredrank2},

\begin{equation*}
\begin{aligned}
A_{GL_2}(q) &= \frac{ | Hom(\Gamma , GL_2(\mathbb{F}_q)) |}{| PGL_2(\mathbb{F}_q)|} + (q-1)\frac{q}{2} + (q-1)\frac{q-2}{2}  -(q-1) \sum_{ \{ A | A \neq 1 \} } [b^1_A]_q \\
&= (q-1)\left( 1 + (a-1)(b-1)\frac{(q+2)}{4} + \frac{q}{2} + \frac{q-2}{2} - (a-1)(b-1) \right) \\
&= (q-1)\left( q + (a-1)(b-1)\frac{(q-2)}{4} \right).
\end{aligned}
\end{equation*}
So in particular, every irreducible components has dimension at most $2$ and the number of $2$-dimensional irreducible components is $1 + \frac{1}{4}(a-1)(b-1)$.\\

Next suppose that $a$ is even and $b$ is odd. Then:

\begin{equation*}
\begin{aligned}
\frac{|Hom(\Gamma , GL_2(\mathbb{F}_q))|}{|GL_2(\mathbb{F}_q)|} &= \sum_{\chi} \nu_a(\chi) \nu_b(\chi) \\
&= \frac{1}{2} \sum_{\alpha \neq \beta} \left( \delta_{\alpha^a} \delta_{\beta^a} + \delta_{\alpha^{a/2}\beta^{a/2}} + \frac{(a-2)}{2}\delta_{\alpha^a \beta^a} \right)
\left( \delta_{\alpha^b} \delta_{\beta^b} + \frac{(b-1)}{2}\delta_{\alpha^b \beta^b} \right) \\
& \; \; \; \; \; + 1 + \frac{a(b-1)}{2}\\
& \; \; \; \; \;  + \frac{1}{2} \sum_{\omega \neq \omega^q} \left( -\delta_{\omega^a} + \delta_{\omega^{(q+1)a/2}} + \frac{(a-2)}{2}\delta_{\omega^{(q+1)a}} \right)
\left( -\delta_{\omega^b} + \frac{(b-1)}{2}\delta_{\omega^{(q+1)b}} \right)\\
&= 1 + \frac{a(b-1)}{4}(q+1).
\end{aligned}
\end{equation*}

\noindent Hence,

\begin{equation*}
\begin{aligned}
A_{GL_2}(q) &= (q-1)\left( 1 + \frac{a(b-1)}{4}(q+1) + \frac{q}{2} + \frac{q-2}{2} -(a-1)(b-1) \right) \\
&= (q-1)\left( q + (b-1)\frac{(aq-3a+4)}{4} \right).
\end{aligned}
\end{equation*}

\noindent In particular, the number of $2$-dimensional irreducible components is $1 + \frac{1}{4}a(b-1)$.
\vspace{0.2cm}

\subsection{Case of $SL_2$}

First, suppose $a$ and $b$ are both odd. By Equation (\ref{equhomtk}) and the Frobenius-Schur indicators for $SL_2(\mathbb{F}_q)$, we get:
\begin{equation*}
\begin{aligned}
\frac{|Hom(\Gamma , SL_2(\mathbb{F}_q))|}{|SL_2(\mathbb{F}_q)|} &= \sum_{\chi} \nu_a(\chi) \nu_b(\chi) \\
&= \frac{1}{2} \sum_{\alpha \neq 1,\alpha_0} \left( \delta_{\alpha^a} + \frac{(a-1)}{2}(1+\alpha(-1)) \right) \left( \delta_{\alpha^b} + \frac{(b-1)}{2}(1+\alpha(-1)) \right) \\
& \; \; \; \; \; + 2\left( \frac{a-1}{2} \right) \left( \frac{b-1}{2} \right) + \frac{1}{2} \sum_{\omega \neq 1,\omega_0} 2 \left( \frac{a-1}{2} \right) \left( \frac{b-1}{2} \right) (1 + \omega(-1)) \\
& \; \; \; \; \; + 1 + (a-1)(b-1) \\
&= 1 + \frac{1}{2}(a-1)(b-1)(q+2).
\end{aligned}
\end{equation*}

\noindent Thus by Theorem \ref{thmcountingredrank2},

\begin{equation*}
\begin{aligned}
A_{SL_2}(q) &= \frac{ | Hom(\Gamma , SL_2(\mathbb{F}_q)) |}{|PGL_2(\mathbb{F}_q)|} + \frac{q}{2}+ \frac{q-2}{2}  - \sum_{ \{ A | A^2 \neq 1 \} }[b^1_{A^2}]_q \\
&= 1 + \frac{1}{2}(a-1)(b-1)(q+2) + (q-1) - \sum_{ \{ A | A^2 \neq 1 \} } [b^1_{A^2}]_q \\
&= q + \frac{1}{2}(a-1)(b-1)(q+2) - 2(a-1)(b-1) \\
&= q + \frac{1}{2}(a-1)(b-1)(q-2).
\end{aligned}
\end{equation*}
In particular, this implies that the irreducible components of the character variety have at most dimension $1$ and that the number of $1$-dimensional irreducible components is $1 + \frac{1}{2}(a-1)(b-1)$. This agrees with the results of \cite{mun}.\\

Now we suppose $a$ is even and $b$ is odd. For this calculation we will also need to assume $q = 1 \; ({\rm mod} \; 2a)$. Then: 
\begin{equation*}
\begin{aligned}
\frac{|Hom(\Gamma , SL_2(\mathbb{F}_q))|}{|SL_2(\mathbb{F}_q)|} &= \sum_{\chi} \nu_a(\chi) \nu_b(\chi) \\
&= \frac{1}{2} \sum_{\alpha \neq 1,\alpha_0} \left( \delta_{\alpha^a} -1 + \frac{a}{2}(1+\alpha(-1)) \right) \left( \delta_{\alpha^b} + \frac{(b-1)}{2}(1+\alpha(-1)) \right) \\
& \; \; \; \; \; + 2\left( \frac{a}{2} \right) \left( \frac{b-1}{2} \right) + \frac{1}{2} \sum_{\omega \neq 1,\omega_0} (a-1) \left( \frac{b-1}{2} \right) (1 + \omega(-1)) \\
& \; \; \; \; \; + 1 + (a-1)(b-1) \\
&= (a-1)(b-1)\left( \frac{q-1}{4} \right) -\frac{(b-1)}{2} + \frac{a(b-1)}{2} \\
& \; \; \; \; \; + (a-1)(b-1)\left( \frac{q-1}{4} \right) + 1 + (a-1)(b-1) \\
& = 1 + \frac{1}{2}(a-1)(b-1)(q+2).
\end{aligned}
\end{equation*}
So we again have that $A_{SL_2}(q) = q + \frac{1}{2}(a-1)(b-1)(q-2)$ and that the number of $1$-dimensional irreducible components is $1+\frac{1}{2}(a-1)(b-1)$. Again, this is in agreement with \cite{mun}.

\section{$SL_3$, $GL_3$ character varieties for fundamental groups of compact oriented surfaces}\label{seccomputgl3sl3}

Let $\Gamma = \pi_1(\Sigma_g)$, where $\Sigma_g$ is a compact oriented surface of genus $g \ge 1$.

\subsection{Case of $GL_3$}
We assume that $q = 1 \; ({\rm mod} \; 6)$. By (\ref{equcounting}) and the calculations in Section \ref{seccasegl3sl3}, we have:

\begin{equation*}
\begin{aligned}
\frac{A_{GL_3}(q)}{m_{q-1}} &= \frac{|Hom(\Gamma, GL_3(\mathbb{F}_q))|}{m_{q-1}|PGL_3(\mathbb{F}_q)|} + \left( 1 - \frac{1}{q-1} -2[4g-4]_q \right)|\mathcal{X}_2^{(1)}| \\
& \; \; \; \; \; + \frac{1}{6}\left( 1 - \frac{1}{(q-1)^2}\right) (m_{q-1}^2-3m_{q-1}+2)\\
& \; \; \; \; \; + \frac{1}{2}\left( 1 - \frac{1}{(q^2-1)}\right)\left( m_{q^2-1} - m_{q-1}\right) + \frac{1}{3}\left( 1 - \frac{1}{(q^2+q+1)}\right)\left( \frac{m_{q^3-1}}{m_{q-1}}-1\right) \\
& \; \; \; \; \; + \left( 1 - \frac{1}{q(q+1)(q-1)^2}\right)(m_{q-1}-1)  +\left(1 - \frac{1}{q^3(q^2-1)(q^3-1)}\right) \\
& \; \; \; \; \; - \frac{(m_{q-1}^2-3m_{q-1}+2)}{(q-1)}[2g-2]_q - 2\frac{(m_{q-1}-1)}{q(q-1)}[2g-2]_q \\
\end{aligned}
\end{equation*}
\begin{equation*}
\begin{aligned}
& \; \; \; \; \; - \frac{(m_{q-1}-1)}{q(q-1)}[2g]_q -  \frac{1}{q^3(q-1)}[2g]_q - (m_{q-1}^2 -3m_{q-1} +2) [2g-2]_q^2 q^{2g-2} \\
& \; \; \; \; \; - 2(m_{q-1}-1) [2g-2]_q [2g]_q q^{2g-3} - (m_{q-1}-1) [2g-2]_q[2g-3]_q q^{2g-1} \\
& \; \; \; \; \; -  [2g]_q [2g-2]_q q^{2g-2} - [2g]_q q^{2g-3} - (m_{q-1}^2 -3m_{q-1} +2) [2g-2]^2_q \\
& \; \; \; \; \; - \frac{2}{q}( m_{q-1} - 1) [2g]_q [2g-2]_q - ( m_{q^2-1} - m_{q-1}) [2g-2]_{q^2} \\
& \; \; \; \; \; - 2\frac{(m_{q-1}-1)}{(q+1)} [2g-2]_q [2g-3]_q - \frac{2}{q^2(q+1)} [2g]_q [2g-1]_q.
\end{aligned}
\end{equation*}

\noindent Collecting some terms and noting that $m_{q^3-1}/m_{q-1} = m_{q^2+q+1}$ gives:

\begin{equation*}
\begin{aligned}
\frac{A_{GL_3}(q)}{m_{q-1}} &= \frac{|Hom(\Gamma, GL_3(\mathbb{F}_q))|}{m_{q-1}|PGL_3(\mathbb{F}_q))|} + \left( 1 - \frac{1}{q-1} -2[4g-4]_q \right)|\mathcal{X}_2^{(1)}| + \frac{1}{6}\left( 1 - \frac{1}{(q-1)^2}\right)m_{q-1}^2 \\
& \; \; \; \; \; + \frac{1}{2}\left( 1 - \frac{1}{(q^2-1)}\right)m_{q^2-1} + \frac{1}{3}\left( 1 - \frac{1}{(q^2+q+1)}\right)m_{q^2+q+1} \\
& \; \; \; \; \;  - \frac{1}{q^3} + \frac{m_{q-1}}{q(q-1)} - \frac{(m_{q-1}^2 -3m_{q-1} +2)}{(q-1)}[2g-2]_q - 2\frac{(m_{q-1}-1)}{q(q-1)}[2g-2]_q \\
& \; \; \; \; \; - \frac{(m_{q-1}-1)}{q(q-1)}[2g]_q -  \frac{1}{q^3(q-1)}[2g]_q - (m_{q-1}^2 -3m_{q-1} +2) [2g-2]_q^2 q^{2g-2} \\
& \; \; \; \; \; - 2(m_{q-1}-1) [2g-2]_q [2g]_q q^{2g-3} - (m_{q-1}-1) [2g-2]_q[2g-3]_q q^{2g-1} \\
& \; \; \; \; \; - [2g]_q [2g-2]_q q^{2g-2} - [2g]_q q^{2g-3} - (m_{q-1}^2 -3m_{q-1} +2) [2g-2]^2_q \\
& \; \; \; \; \; - \frac{2}{q}( m_{q-1} - 1) [2g]_q [2g-2]_q - ( m_{q^2-1} - m_{q-1}) [2g-2]_{q^2} \\
& \; \; \; \; \; - 2\frac{(m_{q-1}-1)}{(q+1)} [2g-2]_q [2g-3]_q - \frac{2}{q^2(q+1)} [2g]_q [2g-1]_q.
\end{aligned}
\end{equation*}

\noindent Further simplifying, we get:

\begin{equation*}
\begin{aligned}
\frac{A_{GL_3}(q)}{m_{q-1}} &= \frac{|Hom(\Gamma, GL_3(\mathbb{F}_q))|}{m_{q-1}|PGL_3(\mathbb{F}_q))|} + \left( 1 - \frac{1}{q-1} -2[4g-4]_q \right)|\mathcal{X}_2^{(1)}| + \frac{1}{6}\left( 1 - \frac{1}{(q-1)^2}\right)m_{q-1}^2 \\
& \; \; \; \; \; + \frac{1}{2}\left( 1 - \frac{1}{(q^2-1)}\right)m_{q^2-1} + \frac{1}{3}\left( 1 - \frac{1}{(q^2+q+1)}\right)m_{q^2+q+1} \\
& \; \; \; \; \;   -q^{6g-6} + \frac{m_{q-1}}{q(q-1)} - \frac{(m_{q-1}^2 -3m_{q-1})}{(q-1)}[2g-2]_q - 2\frac{m_{q-1}}{q(q-1)}[2g-2]_q \\
& \; \; \; \; \; - \frac{m_{q-1}}{q(q-1)}[2g]_q - (m_{q-1}^2 -3m_{q-1}) [2g-2]_q^2 q^{2g-2} - 2m_{q-1} [2g-2]_q [2g]_q q^{2g-3} \\
& \; \; \; \; \; - m_{q-1} [2g-2]_q[2g-3]_q q^{2g-1} - (m_{q-1}^2 -3m_{q-1} ) [2g-2]^2_q \\
& \; \; \; \; \; - \frac{2}{q}m_{q-1} [2g]_q [2g-2]_q - ( m_{q^2-1} - m_{q-1}) [2g-2]_{q^2} \\
& \; \; \; \; \; - 2\frac{m_{q-1}}{(q+1)} [2g-2]_q [2g-3]_q.
\end{aligned}
\end{equation*}

Recall that $|\mathcal{X}_2^{(1)}|$ is given by:
\begin{equation}\label{equx21}
|\mathcal{X}_2^{(1)}| = \frac{|Hom(\Gamma, GL(2,\mathbb{F}_q))|}{|PGL(2,\mathbb{F}_q)|} -\frac{m_{q-1}^2}{2(q-1)} -\frac{m_{q^2-1}}{2(q+1)} -m_{q-1}^2[2g-2]_q -m_{q-1}q^{2g-2}. 
\end{equation}

So we get:
\begin{equation*}
\begin{aligned}
\frac{A_{GL_3}(q)}{m_{q-1}} &= \frac{|Hom(\Gamma, GL_3(\mathbb{F}_q))|}{m_{q-1}|PGL_3(\mathbb{F}_q))|} + \left( 1 - \frac{1}{q-1} -2[4g-4]_q \right)\frac{|Hom(\Gamma, GL(2,\mathbb{F}_q))|}{|PGL(2,\mathbb{F}_q)|} \\
& \; \; \; \; \; + \frac{1}{3}\left( 1 - \frac{1}{(q^2+q+1)}\right)m_{q^2+q+1} +\frac{1}{2}\left( 1 - \frac{1}{(q+1)} \right)m_{q^2-1} -q^{6g-6} \\
& \; \; \; \; \; + \left( \frac{1}{6} -\frac{1}{2(q-1)} + \frac{1}{3(q-1)^2} + \frac{1}{(q-1)} [4g-4]_q q^{2g-2} -[2g-2]_q \right) m_{q-1}^2 \\
& \; \; \; \; \; +\left( q^{4g-6} - q^{2g-2} -q^{2g-4} + [2g-2]_q[2g-1]_q + [2g-2]_q[2g-4]_q \right) m_{q-1}.
\end{aligned}
\end{equation*}

Substituting (\ref{equhomintogl3}) and (\ref{equhomintogl2}) into the above, we get:
\begin{equation*}
\begin{aligned}
\frac{A_{GL_3}(q)}{(q-1)^{2g}} &= \frac{(q^2+q)}{3} \left( q^5+q^4+q^3 \right)^{2g-2} + \frac{(q^2-q)}{2} \left( q^5-q^3 \right)^{2g-2} \\
& \; \; \; \; \; + \left( q^8-q^6-q^5+q^3 \right)^{2g-2} + \left(  q^6-q^5-q^3+q^2 \right)^{2g-2} \\
& \; \; \; \; \; + \left( q^5-q^3-q^2+1 \right)^{2g-2} + (q-2)\left( q^6-q^5-q^4+q^3 \right)^{2g-2} \\
& \; \; \; \; \; + (q-2)\left( q^5-q^4-q^3+q^2 \right)^{2g-2} + \frac{(q-2)(q-3)}{6}\left( q^5-2q^4+q^3 \right)^{2g-2} \\
& \; \; \; \; \; + \left( q-2q^{4g-4} \right)\left(\frac{(q-2)}{2}(q^2-q)^{2g-2}(q-1)^{2g-1} + \frac{q}{2}(q^2+q)^{2g-2}(q-1)^{2g-1}\right) \\
& \; \; \; \; \; + \left( q-2q^{4g-4} \right)\left((q^3-q)^{2g-2}(q-1)^{2g-1} + (q^2-1)^{2g-2}(q-1)^{2g-1}\right) \\
& \; \; \; \; \; + \frac{1}{3}(q^2+q)(q^2+q+1)^{2g-1} +\frac{1}{2}(q^2-q)(q^2-1)^{2g-1} -q^{6g-6} \\
& \; \; \; \; \; + \left( \frac{(q^2+q)}{6} + q^{6g-6}-q^{2g-1} \right) (q-1)^{4g-2} \\
& \; \; \; \; \; +\left( q^{4g-6} - q^{2g-2} -q^{2g-4}\right)(q-1)^{2g} + (q^{2g-2}-1)(q^{2g-4}+q^{2g-1}-2) (q-1)^{2g-2}.
\end{aligned}
\end{equation*}

In particular, $A_{GL_3}(q)$ is a polynomial, so it gives the $E$-polynomial of the $GL_3(\mathbb{C})$-character variety.

\subsection{Case of $SL_3$}

We continue to assume $q = 1 \; ({\rm mod} \; 6)$. From (\ref{equcounting}) and the results of Section \ref{seccasegl3sl3}, we get:

\begin{equation*}
\begin{aligned}
A_{SL_3}(q) &= \frac{|Hom(\Gamma, SL_3(\mathbb{F}_q))|}{|PGL_3(\mathbb{F}_q)|} + \left( 1 - \frac{1}{q-1} -2[4g-4]_q \right)|\mathcal{X}_2^{(1)}| \\
& \; \; \; \; \; + \frac{1}{6}\left( 1 - \frac{1}{(q-1)^2}\right)
\left(m_{q-1}^2-3m_{q-1}+2m_3\right) \\
& \; \; \; \; \; + \frac{1}{2}\left( 1 - \frac{1}{(q^2-1)}\right)\left( m_{q^2-1} - m_{q-1}\right) + \frac{1}{3}\left( 1 - \frac{1}{(q^2+q+1)}\right)\left( m_{q^2+q+1}-m_3\right) \\
& \; \; \; \; \; + \left( 1 - \frac{1}{q(q+1)(q-1)^2}\right)(m_{q-1}-m_3)  +\left(1 - \frac{1}{q^3(q^2-1)(q^3-1)}\right)m_3 \\
& \; \; \; \; \; - \frac{(m_{q-1}^2 -3m_{q-1} +2m_3)}{(q-1)}[2g-2]_q - 2\frac{(m_{q-1}-m_3)}{q(q-1)}[2g-2]_q \\
& \; \; \; \; \; - \frac{(m_{q-1}-m_3)}{q(q-1)}[2g]_q -  \frac{m_3}{q^3(q-1)}[2g]_q - (m_{q-1}^2 -3m_{q-1} +2m_3) [2g-2]_q^2 q^{2g-2} \\
& \; \; \; \; \; - 2(m_{q-1}-m_3) [2g-2]_q [2g]_q q^{2g-3} - (m_{q-1}-m_3) [2g-2]_q[2g-3]_q q^{2g-1} \\
& \; \; \; \; \; - m_3 [2g]_q [2g-2]_q q^{2g-2} - m_3 [2g]_q q^{2g-3} - (m_{q-1}^2 -3m_{q-1} +2m_3) [2g-2]^2_q \\
& \; \; \; \; \; - \frac{2}{q}( m_{q-1} - m_3) [2g]_q [2g-2]_q - ( m_{q^2-1} - m_{q-1}) [2g-2]_{q^2} \\
& \; \; \; \; \; - 2\frac{(m_{q-1}-m_3)}{(q+1)} [2g-2]_q [2g-3]_q - \frac{2}{q^2(q+1)}m_3 [2g]_q [2g-1]_q.
\end{aligned}
\end{equation*}
Collecting some $m_3$ terms and $m_{q-1}$ terms gives:
\begin{equation*}
\begin{aligned}
A_{SL_3}(q) &= \frac{|Hom(\Gamma, SL_3(\mathbb{F}_q))|}{|PGL_3(\mathbb{F}_q))|} + \left( 1 - \frac{1}{q-1} -2[4g-4]_q \right)|\mathcal{X}_2^{(1)}| + \frac{1}{6}\left( 1 - \frac{1}{(q-1)^2}\right)m_{q-1}^2 \\
& \; \; \; \; \; + \frac{1}{2}\left( 1 - \frac{1}{(q^2-1)}\right)m_{q^2-1} + \frac{1}{3}\left( 1 - \frac{1}{(q^2+q+1)}\right)m_{q^2+q+1} \\
& \; \; \; \; \;  - \frac{m_3}{q^3} + \frac{m_{q-1}}{q(q-1)} - \frac{(m_{q-1}^2 -3m_{q-1} +2m_3)}{(q-1)}[2g-2]_q - 2\frac{(m_{q-1}-m_3)}{q(q-1)}[2g-2]_q \\
& \; \; \; \; \; - \frac{(m_{q-1}-m_3)}{q(q-1)}[2g]_q -  \frac{m_3}{q^3(q-1)}[2g]_q - (m_{q-1}^2 -3m_{q-1} +2m_3) [2g-2]_q^2 q^{2g-2} \\
& \; \; \; \; \; - 2(m_{q-1}-m_3) [2g-2]_q [2g]_q q^{2g-3} - (m_{q-1}-m_3) [2g-2]_q[2g-3]_q q^{2g-1} \\
& \; \; \; \; \; - m_3 [2g]_q [2g-2]_q q^{2g-2} - m_3 [2g]_q q^{2g-3} - (m_{q-1}^2 -3m_{q-1} +2m_3) [2g-2]^2_q \\
& \; \; \; \; \; - \frac{2}{q}( m_{q-1} - m_3) [2g]_q [2g-2]_q - ( m_{q^2-1} - m_{q-1}) [2g-2]_{q^2} \\
& \; \; \; \; \; - 2\frac{(m_{q-1}-m_3)}{(q+1)} [2g-2]_q [2g-3]_q - \frac{2}{q^2(q+1)}m_3 [2g]_q [2g-1]_q.
\end{aligned}
\end{equation*}

The $m_3$ terms when collected give $-q^{6g-6}m_3$, so

\begin{equation*}
\begin{aligned}
A_{SL_3}(q) &= \frac{|Hom(\Gamma, SL_3(\mathbb{F}_q))|}{|PGL_3(\mathbb{F}_q))|} + \left( 1 - \frac{1}{q-1} -2[4g-4]_q \right)|\mathcal{X}_2^{(1)}| + \frac{1}{6}\left( 1 - \frac{1}{(q-1)^2}\right)m_{q-1}^2 \\
& \; \; \; \; \; + \frac{1}{2}\left( 1 - \frac{1}{(q^2-1)}\right)m_{q^2-1} + \frac{1}{3}\left( 1 - \frac{1}{(q^2+q+1)}\right)m_{q^2+q+1} \\
& \; \; \; \; \;   -q^{6g-6}m_3 + \frac{m_{q-1}}{q(q-1)} - \frac{(m_{q-1}^2 -3m_{q-1})}{(q-1)}[2g-2]_q - 2\frac{m_{q-1}}{q(q-1)}[2g-2]_q \\
& \; \; \; \; \; - \frac{m_{q-1}}{q(q-1)}[2g]_q - (m_{q-1}^2 -3m_{q-1}) [2g-2]_q^2 q^{2g-2} - 2m_{q-1} [2g-2]_q [2g]_q q^{2g-3} \\
& \; \; \; \; \; - m_{q-1} [2g-2]_q[2g-3]_q q^{2g-1} - (m_{q-1}^2 -3m_{q-1} ) [2g-2]^2_q \\
& \; \; \; \; \; - \frac{2}{q}m_{q-1} [2g]_q [2g-2]_q - ( m_{q^2-1} - m_{q-1}) [2g-2]_{q^2} \\
& \; \; \; \; \; - 2\frac{m_{q-1}}{(q+1)} [2g-2]_q [2g-3]_q.
\end{aligned}
\end{equation*}

\noindent From (\ref{equx21}), we get:

\begin{equation*}
\begin{aligned}
A_{SL_3}(q) &= \frac{|Hom(\Gamma, SL_3(\mathbb{F}_q))|}{|PGL_3(\mathbb{F}_q))|} + \left( 1 - \frac{1}{q-1} -2[4g-4]_q \right)\frac{|Hom(\Gamma, GL(2,\mathbb{F}_q))|}{|PGL(2,\mathbb{F}_q)|} \\
& \; \; \; \; \; + \frac{1}{3}\left( 1 - \frac{1}{(q^2+q+1)}\right)m_{q^2+q+1} +\frac{1}{2}\left( 1 - \frac{1}{(q+1)} \right)m_{q^2-1} -q^{6g-6}m_3 \\
& \; \; \; \; \; + \left( \frac{1}{6} -\frac{1}{2(q-1)} + \frac{1}{3(q-1)^2} + \frac{1}{(q-1)} [4g-4]_q q^{2g-2} -[2g-2]_q \right) m_{q-1}^2 \\
& \; \; \; \; \; +\left( q^{4g-6} - q^{2g-2} -q^{2g-4} + [2g-2]_q[2g-1]_q + [2g-2]_q[2g-4]_q \right) m_{q-1}.
\end{aligned}
\end{equation*}

\noindent Substituting (\ref{equhomintosl3}) and (\ref{equhomintogl2}) into the above, we get:

\begin{equation*}
\begin{aligned}
A_{SL_3}(q) &= \left( 2 \cdot 3^{2g-1} + \frac{(q-1)(q+2)}{3}\right) \left( q^5+q^4+q^3 \right)^{2g-2} + \frac{(q^2-q)}{2} \left( q^5-q^3 \right)^{2g-2} \\
& \; \; \; \; \; + \left( q^8-q^6-q^5+q^3 \right)^{2g-2} + \left(  q^6-q^5-q^3+q^2 \right)^{2g-2} \\
& \; \; \; \; \; + \left( q^5-q^3-q^2+1 \right)^{2g-2} + (q-2)\left( q^6-q^5-q^4+q^3 \right)^{2g-2} \\
& \; \; \; \; \; + (q-2)\left( q^5-q^4-q^3+q^2 \right)^{2g-2} + \left( 3^{2g-1} + \frac{(q-1)(q-4)}{6} \right) \left( q^5-2q^4+q^3 \right)^{2g-2} \\
& \; \; \; \; \; + (q-2q^{4g-4})\left( \frac{(q-2)}{2}(q^2-q)^{2g-2}(q-1)^{2g-1} + \frac{q}{2}(q^2+q)^{2g-2}(q-1)^{2g-1}\right) \\
& \; \; \; \; \; + (q-2q^{4g-4})\left( (q^3-q)^{2g-2}(q-1)^{2g-1} + (q^2-1)^{2g-2}(q-1)^{2g-1} \right) \\
& \; \; \; \; \; + \frac{1}{3}(q^2+q)(q^2+q+1)^{2g-1} +\frac{1}{2}(q^2-q)(q^2-1)^{2g-1} -q^{6g-6} 3^{2g} \\
& \; \; \; \; \; + \left( \frac{(q^2+q)}{6} + q^{6g-6}-q^{2g-1} \right) (q-1)^{4g-2} \\
& \; \; \; \; \; +\left( q^{4g-6} - q^{2g-2} -q^{2g-4}\right)(q-1)^{2g} + (q^{2g-2}-1)(q^{2g-4}+q^{2g-1}-2) (q-1)^{2g-2}.
\end{aligned}
\end{equation*}
This is a polynomial, so it gives the $E$-polynomial of the corresponding $SL_3(\mathbb{F}_q)$-character variety. Setting $q=1$, we obtain:
\begin{theorem}
For $g>1$, the Euler characteristic of $Rep(\pi_1(\Sigma_g) , SL_3(\mathbb{C}))$ is $2 \cdot 3^{4g-3} - 7 \cdot 3^{2g-2}$.
\end{theorem}

\bibliographystyle{amsplain}

\begin{thebibliography}{99}
\bibitem{cala}S. Cavazos, S. Lawton, $E$-polynomial of $SL2(\mathbb{C})$-character varieties of free groups. {\em Internat. J. Math.} {\bf 25} (2014), no. 6, 1450058, 27 pp. 
\bibitem{dakh}V. I. Danilov, A. G. Khovanskii, Newton polyhedra and an algorithm for calculating Hodge-Deligne numbers. {\em Izv. Akad. Nauk SSSR Ser. Mat.} {\bf 50} (1986), no. 5, 925-945. English translation {\em Math. USSR-Izv.} {\bf 29} (1987), no. 2, 279-298.
\bibitem{del1}P. Deligne, Th\'eorie de Hodge. II. {\em Inst. Hautes \'Etudes Sci. Publ. Math.} No. {\bf 40} (1971), 5-57. 
\bibitem{del2}P. Deligne, Th\'eorie de Hodge. III. {\em Inst. Hautes \'Etudes Sci. Publ. Math.} No. {\bf 44} (1974), 5-77. 
\bibitem{dimi}F. Digne, J. Michel, {\em Representations of finite groups of Lie type}. London Mathematical Society Student Texts, {\bf 21}. Cambridge University Press, Cambridge, (1991). iv+159 pp.
\bibitem{eghlsvy}P. G. Etingof, O. Golberg, S. Hensel, T. Liu, A. Schwendner, D. Vaintrob, E. Yudovina, {\em Introduction to representation theory. With historical interludes by S. Gerovitch}. Student Mathematical Library, 59. American Mathematical Society, Providence, RI, (2011). viii+228 pp.
\bibitem{gre}J. A. Green, The characters of the finite general linear groups. {\em Trans. Amer. Math. Soc.} {\bf 80} (1955), 402-447. 
\bibitem{harv}T. Hausel, F. Rodriguez-Villegas, Mixed Hodge polynomials of character varieties. With an appendix by N. M. Katz. {\em Invent. Math.} {\bf 174} (2008), no. 3, 555-624. 
\bibitem{isa}I. M. Isaacs, {\em Character theory of finite groups}. Pure and Applied Mathematics, No. 69. Academic Press, New York-London, (1976), xii+303 pp. 
\bibitem{kagr}M. T. Karkar, J. A. Green. A theorem on the restriction of group characters, and its application to the character theory of $SL(n,q)$. {\em Math. Ann.} {\bf 215} (1975), 131-134. 
\bibitem{lau}G. Laumon, Comparaison de caract\'eristiques d'Euler-Poincar\'e en cohomologie l-adique. {\em C. R. Acad. Sci. Paris S\'er. I Math.} {\bf 292} (1981), no. 3, 209-212. 
\bibitem{leh}G. I. Lehrer, The characters of the finite special linear groups. {\em J. Algebra} {\bf 26} (1973), 564-583. 
\bibitem{lomune}M. Logares, V. Mu\~noz, P. E. Newstead, Hodge polynomials of $SL(2,\mathbb{C})$-character varieties for curves of small genus. {\em Rev. Mat. Complut.} {\bf 26} (2013), no. 2, 635-703.
\bibitem{mac}I. G. Macdonald, {\em Symmetric functions and Hall polynomials}. Second edition. With contributions by A. Zelevinsky. Oxford Mathematical Monographs. Oxford Science Publications. The Clarendon Press, Oxford University Press, New York, (1995). x+475 pp.
\bibitem{mamu1}J. Mart\'inez, V. Mu\~noz, $E$-polynomial of $SL(2,\mathbb{C})$-character varieties of complex curves of genus 3, arXiv:1405.7120 (2014).
\bibitem{mamu2}J. Mart\'inez, V. Mu\~noz, $E$-polynomials of the $SL(2,\mathbb{C})$-character varieties of surface groups, arXiv:1407.6975 (2014).
\bibitem{mart}J. Mart\'inez Mart\'inez, Hodge polynomials of character varieties, PhD thesis, Universidad Complutense de Madrid (2015).
\bibitem{mer}M. Mereb, On the $E$-polynomials of a family of Character Varieties. PhD thesis, arXiv:1006.1286 (2010).
\bibitem{more}S. Mozgovoy, M. Reineke, Arithmetic of character varieties of free groups. {\em Internat. J. Math.} {\bf 26} (2015), no. 12, 1550100, 19 pp. 
\bibitem{mun}V. Mu\~noz, The $SL(2,\mathbb{C})$-character varieties of torus knots. {\em Rev. Mat. Complut.} {\bf 22} (2009), no. 2, 489-497.
\bibitem{ses}C. S. Seshadri, Geometric reductivity over arbitrary base. {\em Advances in Math.} {\bf 26} (1977), no. 3, 225-274. 


\end{thebibliography}

\end{document}